\documentclass[10pt,activeacute,amsfonts,reqno]{amsart}
\usepackage{a4wide}
\usepackage[english]{babel}
\usepackage{inputenc}
\usepackage{mathabx}
\usepackage{subfigure}
\usepackage{amsmath,amsthm,amsfonts,amssymb,amscd}
\usepackage{amssymb,amsfonts}
\usepackage[all,arc]{xy}
\usepackage{enumerate}
\usepackage{mathrsfs}
\usepackage{pgfplots}
\usepackage{marginnote}
\usepackage{mathtools}
\usepackage{amsthm,thmtools,xcolor}

\usepackage[all]{xy}
\usepackage{tikz-cd}
\usetikzlibrary{cd}
\usepackage[square,sort,comma,numbers,nonamebreak]{natbib}

\theoremstyle{plain}
\newtheorem{thm}{Theorem}[section]
\newtheorem{lem}{Lemma}[section]
\newtheorem{cor}{Corollary}[section]
\newtheorem{prop}{Proposition}[section]

\newtheorem{rem}{\textit{Remark}}[section]

\newtheorem{lemx}{Lemma}

\theoremstyle{definition}

\theoremstyle{remark}

\numberwithin{equation}{section}
\usepackage{color}

\usepackage[colorlinks = true,
linkcolor = blue,
urlcolor  = blue,
citecolor = blue,
anchorcolor = blue]{hyperref}
\usepackage{hyperref}
\numberwithin{equation}{section}

\DeclareMathOperator{\supp}{\mathrm{supp}}
\DeclareMathOperator{\sech}{\mathrm{sech}}
\DeclareMathOperator{\csch}{\mathrm{csch}}

\newcommand\underrel[3][]{\mathrel{\mathop{#3}\limits_{%
			\ifx c#1\relax\mathclap{#2}\else#2\fi}}}
		\providecommand{\abs}[1]{\left\lvert#1\right\rvert}
		\providecommand{\norm}[1]{\left\lVert#1\right\rVert}

\title[Principal Chiral Equation]{Global existence and long time behavior in the 1+1 dimensional Principal Chiral model with applications to solitons}

\author[J. Trespalacios]{Jessica Trespalacios$^{\land}$}
\address{Departamento de Ingenier\'{\i}a Matem\'atica, Universidad de Chile, Casilla
	170 Correo 3, Santiago, Chile.}
\email{jtrespalacios@dim.uchile.cl}
\thanks{$^{\land}$ J.T.'s work was funded in part by the National Agency for Research and Development (ANID)/ DOCTORADO NACIONAL/2019 -21190604, Chilean research grants FONDECYT 1191412, Centro de Modelamiento Matemático (CMM), ACE210010 and FB210005, BASAL funds for centers of excellence from ANID-Chile.}


\date{\today}

\begin{document}

	\begin{abstract}
		In this paper, we consider the 1+1 dimensional vector valued Principal Chiral Field model (PCF) obtained as a simplification of the Vacuum Einstein Field equations under the Belinski-Zakharov symmetry. PCF is an integrable model, but a rigorous description of its evolution is far from complete. Here we provide the existence of local solutions in a suitable chosen energy space, as well as small global smooth solutions under a certain non degeneracy condition. We also construct virial functionals which provide a clear description of decay of smooth global solutions inside the light cone. Finally, some applications are presented in the case of PCF solitons, a first step towards the study of its nonlinear stability.
	\end{abstract}
	
	\maketitle
	
	%\tableofcontents
	
	\section{Introduction and main results}
	
	\subsection{Setting} The \textit{Einstein field vacuum} equations and its consequences are key in the Physics of the past century. For a gravitational Lorentzian field $\widetilde g=\widetilde g_{\mu\nu}$ of local signature $(-1,1,1,1)$, one seeks for solving the vanishing of the Ricci tensor 
	\begin{equation}\label{EV}
		\text{R}_{\mu\nu} (\widetilde g)=0.
	\end{equation}
 This equation can be seen as a system of nonlinear quasilinear wave equations. Its importance lies in the fact that many of the characteristic features of the dynamics of the Einstein field equations, are already present in the study of the vacuum equations. See Wald \cite{wald2010general} for a detailed description of the Physics behind these equations. 
 
\medskip
 Under certain symmetries and assumptions, the Einstein field equation can be identified and reduced to the integrable  \emph{Symmetric Principal Chiral Field Equation}, 
\begin{equation}\label{eq:Chiral_0}
		\partial_t\left(\partial_t g g^{-1}\right) -\partial_x\left(  \partial_x g g^{-1} \right)=0, \quad (t,x)\in \mathbb R\times \mathbb R,
	\end{equation}
valid for a $2\times 2$ Riemannian metric $g$. This last equation will be the main subject of this work. This equation is compatible with a certain class of constraints on the metric $g$ that effectively ``reduce'' the equation (\refeq{eq:Chiral_0}) to a system of quasilinear wave equations. 
We will prove existence of local solutions, global small solutions, and describe in part the asymptotic behavior of globally defined solutions. The Principal Chiral Field is a nonlinear $\sigma$-model the which is related to various classical
spinor fields and received huge attention in the 1980s and 1990s. The first description of the integrability of this model in the language of the commutative representation (\refeq{eq:Chiral_0}) was given in \cite{Zakharov1979}, subsequently, different results associated with integrability, conserved quantities and soliton solutions  were obtained \cite{Novikov, faddeev1986integrability, beggs1990solitons}, as well as different analyses of this equation using Backlund transformation, Darboux transformation \cite{haider2008, devchand1998hidden}. In the literature there are several results associated with the study of the reduction of the principal chiral field equation in homogeneous spaces of Lie groups. In particular, Zakharov and Mikhailov in \cite{Zakharov1978}  studied the model of the Principal Chiral Field for the special unitary group $\hbox{SU}(N)$, as well as in \cite{zakharov1980int}, they studied the connection of this equation with the Nambu-Jona-Lasinian model.  In this work, we study a particular case of the reduction problem on ``symmetric spaces'' such as the work of \cite{belinskii1978integration, belinski1979stationary, yaronhadad_2013}. The symmetric space considered is the invariant manifold of symmetric matrices sitting in the Lie group $\hbox{SL}(2;\mathbb{R})$. This space is not a Lie group, but  it can be identified with a Hyperboloid in Minkowski spacetime, see \cite{Matzner1967}.

\medskip

In order to explain the emergence of \eqref{eq:Chiral_0} starting from \eqref{EV}, one needs to consider the so-called Belinski-Zakharov symmetry ansatz \cite{belinskii1978integration}. Symmetry has been a successful method for understanding complicated dynamics in a series of works related to dispersive models, see e.g. \cite{fustero1986einstein, Einstein1937, silva2019scaling}. On the other hand, this assumption is not restrictive, in the sense that several important cases of physical Einstein vacuum metrics are contained under this restriction.
	
\subsection{Belinski-Zakharov spacetimes} 
	Belinski and Zakharov recalled the particular case in which the metric tensor $\widetilde g_{\mu \nu}$ depends on two variables only,  which correspond to spacetimes that admit two commuting Killing vector fields, i.e. an Abelian two-parameter group of isometries, \cite{belinski2001gravitational, belinskii1978integration}. The metric depend on a time-like coordinate $x^0$, and one space-like coordinate $x^1$ (possibly nonnegative). This choice, as will stay clear below, corresponds to considering non-stationary gravitational fields and was first consider  by Kompaneets \cite{kompaneets}. In the particular case that one has a diagonal metric this type of spacetime are often referred to as Einsten-Rosen spacetimes and was first considered in 1937 by Einstein and Rosen \cite{Einstein1937}. 
	
	\medskip
	
	In this work we take these variables to be the time-like and the space-like coordinates $x^0 = t$ and $x^1= x$ respectively. In this case the coordinates are typically expressed using Cartesian coordinates in which $x^i \in \{ t,x\}$  with $i\in \{0,1\}$, and  $x^a,x^b \in\{y,z\}$, where the Latin indexes $a,b\in\{2,3\}$. Then the spacetime interval is a simplified block diagonal form:
	\begin{equation}\label{intervalo}
		ds^2=f(t,x)(dx^2-dt^2)+g_{ab}(t,x)dx^adx^b.
	\end{equation}
	Recall that repeated indexes mean sum, following the classical Einstein convention. Here with abuse of notation we denote $g=g_{ab}$. Due to the axioms of general relativity the matrix $g$ must be real and symmetric. As mentioned above, the structure of this metric is not restrictive, since, from the physical point of view, we find many applications that can be described according to (\refeq{intervalo}). Such spacetimes describe cosmological solutions of general relativity, gravitational waves and their interactions. Also
	have many applications in gravitational theory, \cite{belinski2001gravitational}, we can emphasize that these types of spacetimes belong to the classical solutions of the Robinson–Bondi
	plane waves \cite{BONDI1957}, the Einstein–Rosen cylindrical wave solutions and their two polarization
	generalizations, the homogeneous cosmological models of Bianchi types I–VII including the Kasner model \cite{Kasner1921}, the
	Schwarzschild and Kerr solutions, Weyl axisymmetric solutions, etc. For many more contemporary results the reader can refer to \cite{krasinski2006inhomogeneous}. All
		this shows that in spite of its relative simplicity a metric of the type \eqref{intervalo}
		encompasses a wide variety of physically relevant cases.
	
	\medskip
	
	In order to reduce Einstein vacuum equations \eqref{EV}, one needs to compute the Ricci curvature tensor in terms of the components  of the metric $g=g_{ab}$. The consideration of the metric in the form \eqref{intervalo} results in that the components  $\text{R}_{0a}$ and $\text{R}_{3a}$ of the Ricci tensor are identically zero. Therefore, one can see that system of the Einstein vacuum equations \eqref{EV} decomposes into two  sets of equations. The first one follows from equations $\text{R}_{ab} = 0$, this equation can be written as single matrix equation
	\begin{equation}\label{eq:PCE2}
		\partial_t\left(\alpha \partial_t g g^{-1}\right) -\partial_x\left(\alpha  \partial_x g g^{-1} \right)=0, \quad \det g =\alpha ^2.
	\end{equation}
	We shall refer to this equation as the \textit{reduced Einstein equation.} The trace of the equation \eqref{eq:PCE2} reads 
	\begin{equation}\label{onda}
		\partial_{t}^2 \alpha -\partial_{x}^2 \alpha=0.
	\end{equation}
	This is the so-called \emph{trace} equation; the function $\alpha(t,x)$ satisfies the 1D wave equation,  {\color{blue}for details of the derivation of equations \eqref{eq:PCE2} and \eqref{onda} see \cite[p. 11]{belinski2001gravitational} and \cite[pp. 27 and 147]{yaronhadad_2013}}.  The second set of equations expresses the metric coefficient $f(t,x)$ in terms of explicit terms of $\alpha$ and $g$, where $\det \widetilde g_{\mu \nu}:= -f^2\alpha^2 $. For the moment, this expression is not relevant in this introduction, for more details see \cite{belinski2001gravitational}. 
	
\subsection{New coordinates}\label{sec1.3}
	The fact that the $2\times 2$ matrix $g$ is symmetric allows one to diagonalize it for fixed $t$ and $x$. One writes $g= RDR^{T}$, where $D$ is a diagonal matrix and $R$ is a rotation matrix, of the form
	\begin{align}
		D=\left( \begin{array}{cc}
			\alpha e^{\Lambda}  & 0 \\
			0 &\alpha  e^{-\Lambda} 
		\end{array} \right), \qquad  R=\left( \begin{array}{cc}
			\cos \phi   & -\sin \phi \\
			\sin \phi & \cos \phi 
		\end{array} \right).
	\end{align}
Clearly $\det g = \alpha^2$.	Here $\Lambda$ is the scalar field that determines the  eigenvalues of $g$, and the scalar field $\phi$ determines the deviation of $g$ from being a diagonal matrix. Since $\phi$ is considered as an angle, we assume $\phi\in [0,2\pi]$. Therefore $\Lambda, \phi$ and $\alpha$ can be considered as the three degrees of freedom in the symmetric matrix $g$, \cite{yaronhadad_2013}. Written explicitly, the matrix $g$ is given now by
	\begin{equation}\label{diag1}
		g=\alpha\left( \begin{array}{cc}
			\cosh{\Lambda} +\cos 2\phi \sinh{\Lambda} &  \sin 2\phi \sinh{\Lambda} \\
			\sin 2\phi \sinh{\Lambda}     &  \cosh{\Lambda}-\cos 2\phi \sinh{\Lambda} 
		\end{array} \right).
	\end{equation}
Some analog representations have been used in various results associated, for example to the Einstein-Rosen metric \cite{carmeli1984einstein}.	Note that Minkowski $g_{\mu\nu}=(-1,1,1,1)$ can be recovered by taking $\Lambda=0$, $\alpha=1$ and $\phi$ free.  
	Now, with this representation, the equation (\refeq{eq:PCE2}) read
	\begin{equation}\label{sistema1}
		\begin{cases}
			\partial_t(\alpha \partial_t\Lambda) - \partial_x(\alpha \partial_x\Lambda) = 2\alpha  \sinh{2\Lambda}((\partial_t\phi)^2-(\partial_x\phi)^2), \\
						\partial_t(\alpha \sinh^2 \Lambda \partial_t\phi )- \partial_x(\alpha \sinh^2 \Lambda \partial_x\phi )=  0,\\
		\partial_{t}^2\alpha-\partial_{x}^2\alpha= 0,
		\end{cases}
	\end{equation}
and 
\begin{equation}\label{eq:f} 
	\partial_{t}^2(\ln f)- \partial_{x}^2(\ln f) = G,
	\end{equation}
	where $G=G[\Lambda, \phi, \alpha]$ is given by 
	\begin{equation}\label{sistema2}
		\begin{aligned}
			& ~{} G := - \left( \partial_{t}^2(\ln \alpha)-\partial_{x}^2(\ln \alpha) \right)- \dfrac{1}{2\alpha^2} ((\partial_t\alpha)^2-(\partial_x\alpha)^2) \\
			&~{} \qquad  - \dfrac{1}{2}((\partial_t\Lambda)^2-(\partial_x\Lambda)^2)- 2\sinh^2 \Lambda((\partial_t\phi)^2-(\partial_x\phi)^2).%= 0.
		\end{aligned}
	\end{equation}
	Note that the equation for $\alpha$ is the standard one dimensional wave equation, and can be solved independently of the other variables. Also, given $\alpha$, $\Lambda $ and $\phi$, solving for $\ln f$ reduces to use D'{}Alembert formula for linear one dimensional wave with nonzero source term. Consequently, the only nontrivial equations in \eqref{sistema1} are given by
	\begin{equation}\label{sistema3}
	\begin{cases}
			\partial_t(\alpha \partial_t\Lambda) - \partial_x(\alpha \partial_x\Lambda) = 2\alpha  \sinh{2\Lambda}((\partial_t\phi)^2-(\partial_x\phi)^2)\\
				\partial_t(\alpha \sinh^2 \Lambda \partial_t\phi )- \partial_x(\alpha \sinh^2 \Lambda \partial_x\phi )=  0,
			\end{cases}
	\end{equation}
	for $\alpha$ solution to linear 1D wave. Because of the difficulties found dealing with this system, we shall concentrate efforts in a more modest case. If one settles $\alpha\equiv1$ constant, in this case the metric \eqref{intervalo} is diffeomorphic to Minkowski \cite{belinski2001gravitational, yaronhadad_2013}. In this paper, we avoid this oversimplification {\color{blue}\bf by only taking \eqref{sistema3} with $\alpha \equiv 1$}, not considering the function $f$, namely 
	\begin{equation}\label{sistema4}
	 	\begin{cases*}
	 		\partial^2_{t}\Lambda- \partial_x^2\Lambda=-2\sinh(2\Lambda)((\partial_x \phi)^2-(\partial_t\phi)^2),\\
	 		\partial_t^2 \phi-\partial_x^2 \phi=-\dfrac{\sinh(2\Lambda)}{\sinh^2(\Lambda)}(\partial_t \phi \partial_t \Lambda -\partial_x \phi \partial_x \Lambda).
	 	\end{cases*}
	 \end{equation}
	 \eqref{sistema4}  is a set of coupled quasilinear wave equations, with a rich analytical and algebraic structure, as we shall see below. Also, it coincides with \eqref{eq:Chiral_0} under $g$ as in \eqref{diag1} and $\alpha\equiv 1$. Understanding this particular case will be essential to fully understand the general case of non constant $\alpha.$  It is very notable that the basic set of equations of the Einstein equation for the metric (\ref{intervalo}) coincides with the main Chiral Field equation (\refeq{eq:Chiral_0}) when $\alpha$ is constant. If we only consider this set of equations, the principal Chiral Field equation formally admits nontrivial solutions which would correspond to a special subclass of Chiral Field theory solutions, \cite{belinski2001gravitational}. It should be noted that in the particle case where $\alpha$ is constant, the \textit{reduced Einstein equation} \eqref{eq:PCE2} corresponds to the chiral field equation \eqref{eq:Chiral_0}, as mentioned above, however,  as we will see later, from the definitions of energy and momentum densities of the Chiral Field equation, we cannot deduce relevant results from the Einstein field equation when $\alpha$ is an arbitrary function, the non-constant $\alpha$ case requires a different treatment.  As a consequence of the above observation, in the constant $\alpha$ case, the equation \eqref{eq:f} is no longer coupled to the system to be worked.
	 
	{\color{blue} \begin{rem}
	 Notice that the choice $\alpha\equiv 1$ is aso made because the equations \eqref{sistema1} may have a different behavior depending on the properties of the function $\alpha$. Even in this case ($\alpha\equiv 1$), the PCF model is sufficiently rich to produce a complex dynamics. In our recent result \cite{MT2023}, posted online very recently, we consider the more demanding case $\alpha$ non constant, but still under some particular conditions that are natural generalizations of the hypotheses presented here. Finally, the current work has been essential to obtain the general results presented in \cite{MT2023}.
	 	 \end{rem}
	 }

	As we can see from the matrix form (\ref{diag1}) the solutions in terms of the fields $\Lambda$ and $\phi$ are not unique, since these fields satisfy a gauge invariance, that is, 
	\begin{equation}
		(\Lambda,\phi ) \quad \hbox{solution}, \qquad  \left(\Lambda,\phi  + k\pi \right) \quad \hbox{solution}, k\in\mathbb Z.
	\end{equation}
It should be noted that although  \eqref{sistema4} is strictly non-linear in the fields $\Lambda(t,x)$ and $\phi(t,x)$, it has many similitudes with the classical linear wave equation and with Born-Infeld equation \cite{Alejo2018}: given any $\mathcal{C}^2$ real-valued profiles $h(s), k(s)$, then the following functions are solutions for Eqns. (\ref{sistema4})
		\begin{equation}\label{light_cone_solutions}
			\Lambda(t,x)=h(x\pm t)=h(s), \qquad \phi(t,x)=k(x\pm t)=k(s).
		\end{equation}
This property will be key when establishing the connection between the local theory that will be presented in the following section and the analysis of explicit solutions to the equation in the Section \refeq{Sect:5}. System \eqref{sistema4} is a Hamiltonian system, having the conserved energy	
	\begin{equation}\label{Energy}
	\hspace{-0.4cm}	E[\Lambda,\phi](t) := \int
		\left( \dfrac{1}{2}((\partial_{t}\Lambda)^{2}+(\partial_{x}\Lambda)^{2})+2\sinh^{2}{\Lambda}((\partial_{t}\phi)^{2}+(\partial_{x}\phi)^{2}) \right)(t,x)dx.
	\end{equation}
	Note that the energy is well-defined if $(\Lambda,\partial_t\Lambda)\in \dot H^1 \times L^2$, but a suitable space for $(\phi,\partial_t\phi)$ strongly depends on the weight $\sinh^{2}{\Lambda}$, which can easily grow exponential in space, since $ \dot H^1$ can easily contain unbounded functions. In this sense, making sense of $E[\Lambda,\phi](t)$ (even for classical solution such as solitons) is subtle and requires a deep and careful analysis which will be done later.
	
	\medskip
	
	The notion of the energy and the law of conservation of energy play a key role in all mathematics-physical theories. However, the definition  of energy in relativity is a complex matter, and this problem has been given a lot of attention in the literature \cite{wald2010general, wald2000}. The most likely candidate for the energy density for the gravitational field in general relativity would be a quadratic expression in the first derivatives of the components of the metric \cite{wald2010general}. In this case we have a particular structure of spacetime and the equation (\refeq{diag1}) gives us a decomposition  of the metric in terms of the fields $\Lambda$ and $\phi$. 
	
	\medskip

	Coming back to our problem, and using inverse scattering techniques, Belinski and Zakharov \cite{belinskii1978integration} considered \eqref{eq:PCE2} giving a first approach to this problem. They introduce a Lax-pair for \eqref{eq:PCE2}-\eqref{onda}, together with a general method for solving it. Localized structures and multi-coherent were found, but they are not solitons in the standard sense, unless $\alpha$ is constant, a more in-depth study on the subject, is also made in \cite{belinskii1978integration, belinski2001gravitational}. More recently Hadad \cite{yaronhadad_2013} explored the Belinski-Zakharov transformation for the 1+1 Einstein equation. It is used to derive explicit formula for solutions on arbitrary diagonal background, in particular on the Einstein-Rosen background.  
	
	\subsection{Main results}
One of the main purposes of this paper is to give a rigorous description of the dynamics for \eqref{eq:Chiral_0} in the so-called energy space associated to the problem, and close to important exact solutions. We will present three different results: local, global existence, and long time behavior of solutions, in particular solitons.  

	\medskip
	
	Our first result is a classical local existence result for solutions in the energy space.  As mentioned above,  the system (\refeq{sistema4}) is a set a coupled  quasilinear wave equations, with a rich analytical and algebraic structure. 

	 Clearly in the analysis of the initial value problem for this system, we have a component of difficulty related to the regularity of the term $\frac{\sinh(2\Lambda)}{\sinh^2(\Lambda)}$ when the function $\Lambda(t,x)$ is zero, which must be carefully analyzed in order to be able to construct a result of local well-posedness  associated to (\refeq{sistema4}). In order to develop the results related to the local theory for the nonlinear wave equation, let us write the function  $\Lambda(t,x)$ in the form
	 \begin{equation}\label{tilda}
	 \Lambda(t,x):= \lambda + \tilde{\Lambda}(t,x), \quad \lambda \neq 0.
	 \end{equation}
	 Notice that this choice makes sense with the energy in \eqref{Energy}, in the sense that $\Lambda\in \dot H^1$ and $\partial_t\Lambda \in L^2$. Without loss of generality, we assume $\lambda>0$. The basic idea is to establish the conditions that are required on $\lambda$ and $\tilde{\Lambda}$ in order to obtain the desired regularity results. With this choice, the system (\refeq{sistema4}) can be written in terms of the function $\tilde{\Lambda}(t,x)$ as follows: 
	 \begin{equation}
	 	\begin{cases*}
	 		\partial^2_{t}\tilde{\Lambda}- \partial_x^2 \tilde{\Lambda}=-2\sinh(2\lambda +2\tilde{\Lambda})((\partial_x \phi)^2-(\partial_t\phi)^2),\\
	 		\partial_t^2 \phi-\partial_x^2 \phi =-\dfrac{\sinh(2\lambda +2\tilde{\Lambda})}{\sinh^2(\lambda +\tilde{\Lambda})}(\partial_t \phi \partial_t \tilde{\Lambda} -\partial_x \phi \partial_x \tilde{\Lambda}).
	 	\end{cases*}\label{cauchy2}
	 \end{equation}
	 This is the system we are going to analyze along this paper. 
	 
\medskip	 
	 
	 Let us consider the following notation : 
	 \begin{equation}\label{notation}
	 	\begin{cases}\Psi=\left( \tilde{\Lambda},\phi \right),\quad 
	 		\partial \Psi= \left( \partial_t \tilde{\Lambda}, \partial_x \tilde{\Lambda}, \partial_t \phi, \partial_x \phi \right),\\
			\vspace{0,1cm}
	 		\abs{ \partial \Psi }^2= \big| \partial_t \tilde{\Lambda}\big|^2 +\big| \partial_x \tilde{\Lambda} \big|^2+\abs{\partial_t \phi }^2+\abs{\partial_x \phi}^2,\\
			\vspace{0,1cm}
			  F(\Psi,\partial \Psi)= \left( F_1,F_2\right),\\
	 		F_1(\Psi,\partial \Psi):= 2\sinh(2\lambda +2\tilde{\Lambda})\left((\partial_x \phi)^2-(\partial_t\phi)^2\right),\\
			\vspace{0,1cm}
	 		F_2(\Psi,\partial \Psi):= \dfrac{\sinh(2\lambda +2\tilde{\Lambda})}{\sinh^2(\lambda +\tilde{\Lambda})} \left(\partial_t \phi \partial_t \tilde{\Lambda} -\partial_x \phi \partial_x \tilde{\Lambda} \right).
	 	\end{cases}
	 \end{equation}
	 With this notation, the initial value problem for (\refeq{cauchy2}) can be written in vector form as follows
	 \begin{equation}
	 	\begin{cases*}
	 		\partial_{\alpha} (m^{\alpha \beta}\partial_{\beta}\Psi)=F(\Psi,\partial \Psi)\\
	 		(\Psi,\partial_t \Psi)|_{\{t=0\}}=(\Psi_0, \Psi_1) \in  \mathcal{H}.
	 	\end{cases*}\label{IVP}
	 \end{equation}
	 Where $m^{\alpha \beta}$ are the components of the Minkowski metric with $\alpha, \beta \in \left\{0,1\right\}$, and 
	 \begin{equation}\label{Hcal}
	(\Psi,\partial_t \Psi)\in \mathcal{H}:=H^1(\mathbb{R})\times H^{1}(\mathbb{R}) \times L^2(\mathbb{R}) \times L^2(\mathbb{R}).
	 \end{equation}
	 Notice that from \eqref{tilda}, $\Lambda\in \dot H^1$. We are also going to impose the following condition on the initial data  
	 \begin{equation}\label{condicion1}
	 	\norm{\left(\Psi_0,\Psi_1\right)}_{\mathcal{H}} \leq \dfrac{\lambda}{2D},
	 \end{equation}
	 where the assumptions on the constant $D\geq 1$ will be indicated below. An evolution equation is said to be well-posed in the sense of Hadamard, if existence, uniqueness of solutions and continuous dependence on initial data hold. 
	 
	 \medskip
	 
	 The following proposition shows that the equation (\refeq{IVP}), in terms of the function $\tilde{\Lambda}$ introduced in \eqref{tilda}, is locally well-posed in the space $L^{\infty}([0,T]; \mathcal{H})$ with the norm in this space defined by 
	 	\begin{equation*}
	 		\norm{ (\Psi,\partial_t \Psi) }_{L^{\infty}([0,T]; \mathcal{H})} = \sup_{t\in [0,T]} \left( \norm{ \Psi }_{H^1(\mathbb{R})\times H^{1}(\mathbb{R})}+ \norm{ \partial_t \Psi }_{L^2(\mathbb{R})\times L^2(\mathbb{R})} \right),
	 	\end{equation*} 
with $(\Psi,\partial_t \Psi)$ introduced in \eqref{notation}. Our first result is the following.
	 
	 \begin{prop}\label{LOCAL}
	 	If $(\Psi_0, \Psi_{1})$ satisfies the condition  (\ref{condicion1}) with an appropriate constant $D\geq1$, then: 
	 	\begin{itemize}
	 		\item[(1)] (Existence and uniqueness of local-in-time solutions). There exists 
	 		\[ 
			T=T\left( \norm{ \left( \tilde{\Lambda}_0, \phi_0 \right) }_{H^1(\mathbb{R}) \times H^1(\mathbb{R})}, \norm{ \left( \tilde{\Lambda}_1, \phi_1 \right) }_{L^2(\mathbb{R}) \times L^2(\mathbb{R})},\lambda \right) > 0,\]
	 		such that {\color{blue} there exists a solution} $\Psi$ to  (\refeq{IVP}) with 
	 		\begin{equation*}
	 			(\Psi,\partial_t \Psi)\in L^{\infty}([0,T];\mathcal{H}).
	 		\end{equation*}
	 		Moreover, the solution is unique in this function space. {\color{blue} If the data has more regularity, the solution is classical, see Appendix \ref{C}.}
			
			\medskip
			
	 		\item[(2)] (Continuous dependence on the initial data). Let $\Psi_{0}^{(i)}, \Psi_{1}^{(i)}$ be sequence such that $\Psi_{0}^{(i)} \longrightarrow \Psi_{0}$ in $H^1(\mathbb{R})\times H^{1}(\mathbb{R})$ and $\Psi_{1}^{(i)} \longrightarrow \Psi_{1}$ in $L^2(\mathbb{R})\times L^2(\mathbb{R})$ as $i \longrightarrow \infty.$ Then taking $T>0$ sufficiently small, we have 
	 		\begin{equation*}
	 			\norm{ \left(\Psi^{(i)}-\Psi, \partial_t(\Psi^{(i)}-\Psi) \right) }_{L^{\infty}([0,T]; \mathcal{H})} \longrightarrow 0.
	 		\end{equation*}	 
	 		Here $\Psi$ is the solution arising from data $(\Psi_0,\Psi_1)$ and $\Psi^{(i)}$ is the solution arising from data $\left(\Psi_0^{(i)},\Psi_1^{(i)} \right).$
	 	\end{itemize}	 	
	 \end{prop}
	 { \color{blue} Note that the above proposition does not directly give us a classical solution to the problem, however, if it is assumed that the initial data is sufficiently regular, in fact the solution can be understood as classical, see Appendix \ref{C}.}
	 Energy estimates for the wave equation will be key to prove the previous result. The importance of this methodology lies in the fact that it is robust enough to deal with situations in which the solution of the equation may not be explicit. In particular, they are of crucial importance in the study of nonlinear equations, as is our case. Additionally, they allow to obtain decay inequalities for the nonlinear terms of the equation, and provide results on the local existence for certain quasilinear wave equations, generally with small Cauchy data. Global results are also based on energy estimates, the Sobolev theorem, as well as the generalized Klainerman Sobolev inequalities, which make use of vector fields preserving the wave equation. For an exhaustive study of energy estimates for the wave equation see e.g. \cite{sogge}.
	 
	 \medskip
	 
	 Having established the existence of solutions, our second result involves whether or not local solutions can be extended globally in time. This is not an easy problem, mainly because $\Lambda(t,x)$ may achieve the zero value in finite time. Therefore, an important aspect of the proof will be to ensure uniform distance from zero of the function $\Lambda(t,x)$.

	\begin{thm}\label{GLOBAL0}
	Consider the semilinear wave system \eqref{IVP} posed in $\mathbb{R}^{1+1},$ with the following initial conditions:
	\begin{equation}
		\begin{cases}
			(\phi,\tilde{\Lambda})|_{\{t=0\}}= \varepsilon(\phi_0,\tilde{\Lambda}_0), \quad (\phi_0,\tilde{\Lambda}_0)\in C_c^{\infty}(\mathbb{R}\times \mathbb{R}), \\
			(\partial_t\phi,\partial_t\tilde{\Lambda})|_{\{t=0\}}= \varepsilon(\phi_1,\tilde{\Lambda}_1),\quad   (\phi_1,\tilde{\Lambda}_1) \in C_c^{\infty}(\mathbb{R}\times \mathbb{R}).  
		\end{cases}\label{Ciniciales}
	\end{equation} 
Then, there exists $\varepsilon_0$ sufficiently small such that if $\varepsilon < \varepsilon_0$, the unique solution remains smooth for all time and have finite conserved energy \eqref{Energy}.
	\end{thm}
{\color{blue}
The condition that the initial data is compactly supported can be relaxed, but it is essential to have enough decay. For simplicity of exposition, we shall assume that the data is compactly supported, as it is usually done in the literature, see for example \cite{sogge}. 
}
\medskip

The global existence problem stated above a key part of the analysis comes from the fact that \eqref{IVP} can be written as
\begin{equation}
		\begin{cases*}
			\square \tilde{\Lambda}=-2\sinh(2\lambda +2\tilde{\Lambda})Q_0(\phi,\phi),\\
			\square \phi =\dfrac{\sinh(2\lambda +2\tilde{\Lambda})}{\sinh^2(\lambda +\tilde{\Lambda})}Q_0(\phi,\tilde{\Lambda}),
		\end{cases*}\label{problema3}
	\end{equation} 
	where $Q_0$ represents the well-known fundamental null form
\begin{equation}\label{NC}
	Q_0 (\phi,\tilde{\Lambda})= m^{\alpha \beta}\partial_{\alpha} \phi \partial_{\beta} \tilde{\Lambda},
\end{equation}
where $m_{\alpha \beta}$ to denote the standard Minkowski metric on $\mathbb{R}^{1+1}$. The smallness in the initial data implies that the nonlinear equation can be solved over a long period of time and the global solution can be constructed once the non-linearity decays enough. Moreover, the slower decay rate in low dimensions can be compensated by the special structure of the nonlinearity.  
	
	\medskip
	
	Global existence of small solutions to nonlinear wave equations with null conditions has been a subject under active investigation for the past four decades.  The approach to understand the small data problem with null condition was introduced by Klainerman in the pioneering works \cite{klainerman} and by Christodoulou \cite{Christodoulou1986},
for the global existence of classical solutions for nonlinear wave equations with null conditions in three space dimensions.  Alinhac in \cite{alinhac2001null} studied the problem for the case of two space dimensions. We remark here that in $\mathbb{R}^{3+1}$ the null condition is a sufficient but not necessary condition to obtain a small-data-global-existence result, see e.g. \cite{lindblad2008, lindblad2010}. More recently Huneau and Stingo \cite{huneau2021global} studied the global existence for a toy model for the Einstein equations with additional compact dimensions, where the nonlinearity is linear combinations of the classical quadratic null forms. In one space dimension waves do not decay, and nonlinear resonance can lead to finite time blow up. Nevertheless, Luli, Yang and Yu in \cite{Luli2018} proved, for Cauchy problems of semilinear wave equations with null conditions in one space dimension, the global existence of classical solutions with small initial data. The authors proposed a weighted energy and use the bootstrap method for obtain the result. The system in the Theorem (\refeq{GLOBAL0}) does not obey the classical null condition. However, the factors $Q_0(\phi,\phi)$ and $Q_0(\phi,\tilde{\Lambda})$ provide decay and with the appropriate condition on $\lambda$, global regularity can be obtained. 
Inspired by Luli, Yang and Yu’s  result \cite{Luli2018} in the semilinear case, it is natural to conjecture that the Cauchy problem for one-dimension system of quasilinear wave equations (\refeq{problema2})  admits a  global classical solution for small initial data. The main aim of this theorem is to verify this conjecture. 
\medskip

Now we discuss the long time behavior of globally defined solutions. Here, virial identities will be key to the long-time description. 
	\begin{thm}\label{LTD}
	Let $(\Lambda,\phi)$ be a global solution to \eqref{sistema4} such that its energy $E[\Lambda,\phi](t)$ is conserved and finite. Then, for any $v\in (-1,1)$ and
	 $\omega(t)= t/ \log^{2}t$, one has 
	\[
	\lim_{t\to +\infty}\int_{vt -\omega(t)}^{vt+ \omega(t)}
		\left((\partial_{t}\Lambda)^{2}+(\partial_{x}\Lambda)^{2}+\sinh^{2}{\Lambda}((\partial_{t}\phi)^{2}+(\partial_{x}\phi)^{2}) \right)(t,x)dx=0.
	\]
%	Moreover, one has the pointwise estimate of temporal decay, valid for any $t$ sufficiently large, and $C>0$ independent of large time:
%	\begin{equation}\label{rate_decay}
%	\int_{vt -\omega(t)}^{vt+ \omega(t)}
%		\left((\partial_{t}\Lambda)^{2}+(\partial_{x}\Lambda)^{2}+\sinh^{2}{\Lambda}((\partial_{t}\phi)^{2}+(\partial_{x}\phi)^{2}) \right)(t,x)dx \leq \frac{C}{\log t}.
%	\end{equation}
	\end{thm}
	This result establishes that inside the light cone, all finite-energy solutions must converge to zero as time tends to infinity. It is also in concordance with the solutions found in \eqref{light_cone_solutions}, which are a natural counterexample in the case $v=\pm 1.$ A similar outcome has been recently found in \cite{Alejo2018}, where the less involved Born-Infeld model is considered. Note that Theorem \ref{LTD} is valid under general data, and compared with the obtained asymptotic result in Theorem \ref{GLOBAL0}, reveals that the decay property may hold under very general initial data, and unlike \cite{Alejo2018}, our model is in some sense semilinear. 
		
		\medskip
		
		As a final comment on this part of our results, we should mention the work by Yan \cite{WYan2019} dealing with the blow-up description in the Born-Infeld theory. We strongly believe that the blow-up mechanism in the PCF model is triggered by the threshold $\Lambda =0$.

	\subsection{Application to soliton solutions} An important outcome of our previous results is a clear background for the study of soliton solutions of  \eqref{sistema4}. Belinski and Zakharov in \cite{belinskii1978integration} proposed that the Eq. \eqref{eq:PCE2} has $N-$soliton solutions, see also \cite{belinski2001gravitational} for further details. Hadad \cite{yaronhadad_2013} also showed explicit examples of soliton solutions for the equation (\ref{sistema4}) using diagonal backgrounds, also called ``seed metric''. Basically, one starts with a background solution  of the form
	\begin{equation}\label{g0}
		g^{(0)}=\left[\begin{array}{cc}
			e^{\Lambda^{(0)}} & 0\\
			0 & e^{-\Lambda^{(0)}}
		\end{array} \right].
	\end{equation}
	The function $\Lambda^{(0)}(t,x)$ satisfies the wave equation $\partial_t^2\Lambda^{(0)}-\partial_x^2\Lambda^{(0)}=0.$
	In this case, if we want identify the solution in terms of the equation (\refeq{diag1}), we have that $\Lambda=\Lambda^{(0)}$, $\phi=n\pi,$ with $n\in \mathbb{Z}$, and $\alpha=1$. The gauge choice for us will be $n=0$.
	
	\medskip
	
	As expressed in \cite{yaronhadad_2013}, an important case is the one-soliton solution, which is obtained by taking $\Lambda^{(0)}$ time-like and equals to $t$ and $\phi^{(0)}=0$. Note that with this choice the energy is not well-defined, but a suitable modification will make this metric regular again. Indeed,  the energy proposed in (\refeq{Energy}) is not finite, but one can consider the following modified energy 
	\begin{equation}\label{Energia_simpleM}
		E_\text{mod}[\Lambda,\phi](t):=\int \left(\dfrac{1}{2}\left( (\partial_t\Lambda)^2-1+ (\partial_x \Lambda)^2 \right) +2\sinh^2(\Lambda)((\partial_t \phi)^2+(\partial_x \phi)^2)  \right),
	\end{equation}
	which is also conserved and identically zero. Hadad computed the corresponding 1-soliton solution using Belinski and Zakharov techniques, obtaining 
	\begin{equation}\label{soliton}
		g^{(1)}=	\left[\begin{array}{cc}
			\dfrac{e^{t}Q_c(x-vt)}{Q_c(x-vt-x_0)} & -\dfrac{1}{c}Q_c(x-vt)\\
			-\dfrac{1}{c}Q_c(x-vt) & \dfrac{e^{-t}Q_c(x-vt)}{Q_c(x-vt+x_0)}
		\end{array} \right],
	\end{equation}
	where, for a fixed parameter $\mu > 1$, one has 
	\[
	Q_c(\cdot)=\sqrt{c}\sech(\sqrt{c} (\cdot)), \quad c=\left(\dfrac{2\mu}{\mu^2-1}\right)^2, \quad v=-\dfrac{\mu^2+1}{2\mu}<-1, \quad \hbox{and}  \quad x_0=\dfrac{\ln |\mu|}{\sqrt{c}}.
	\]
	Notice that the first component of $g^{(1)}$ grows in time. The parameter $\mu$ represents a pole in terms of scattering techniques, however this point of view will be considered in another work. Therefore, we have a traveling superluminal soliton which travels to the left  (if $\mu > 0$). Also, representing $g^{(1)}$ in terms of corresponding functions $\Lambda^{(1)},\phi^{(1)}$ is complicated, and done in Section \ref{Sect:5}.
	 
\medskip

In this paper, we propose a modification of this ``degenerate'' soliton solution by cutting off the infinite energy part profiting of the wave-like character of solutions $\Lambda^{(0)}$. Although it is not so clear that they are physically meaningful, these new solutions have finite energy and local well-posedness properties in a vicinity.

\medskip

 Indeed, consider a smooth function $\theta \in  C^{2}_c(\mathbb R)$. Additionally, consider the constraint $0<\mu<1$. For any $\lambda>0$, and $\varepsilon >0$ small, let
\[
\Lambda^{(0)}_{\varepsilon} := \lambda + \varepsilon \theta(t+x), \quad \phi^{(0)}: =0.
\]		
Clearly $\Lambda^{(0)}_{\varepsilon}$ solves the wave equation in $1D$ and has finite energy $E[\Lambda^{(0)}_{\varepsilon},\phi^{(0)}_{\varepsilon}]<+\infty$. This will be for us the background seed. The corresponding 1-soliton is now
\begin{equation}\label{solitonG}
	g^{(1)}=	\left[\begin{array}{cc}
		\dfrac{e^{\lambda +\varepsilon\theta}\sech(\beta(\lambda + \varepsilon \theta))}{\sech(\beta(\lambda + \varepsilon \theta)-x_0)} & -\dfrac{1}{\sqrt{c}}\sech(\beta(\lambda + \varepsilon \theta))\\
		-\dfrac{1}{\sqrt{c}}\sech(\beta(\lambda + \varepsilon \theta))& \dfrac{e^{-(\lambda + \varepsilon \theta)}\sech(\beta(\lambda + \varepsilon \theta))}{\sech(\beta(\lambda + \varepsilon \theta)+x_0)}
	\end{array} \right], \qquad \beta=\frac{\mu +1}{\mu -1},
\end{equation} 
which also has finite energy. Perturbations of the fields $\Lambda$  and $\phi$ associated with  this soliton  will be globally defined according to the Theorem \ref{GLOBAL0}:

\begin{cor}\label{aplication}
%Suitable perturbations of the small soliton (\refeq{solitonG})  are globally well-defined. 
Suitable perturbations of any soliton as in \eqref{solitonG} are globally well-defined. 
\end{cor}	

This result has an important outcome: it allows us to try to study the stability of these solutions, which will be done in a forthcoming work. Additionally, there are other possible choices of metrics in Einstein's field equations that lead to the KdV model, see e.g. \cite{sarma2010kdv}.

	\subsection*{Organization of this paper} This paper is organized as follows: Section \ref{Sect:2} is devoted to the proof of local existence of solutions. In Section \ref{Sect:3} we prove global existence of small solutions close to a nonzero value. Section \ref{Sect:4} is devoted to the long time behavior of solutions, and finally Section \ref{Sect:5} consider the particular case of solitons.
	
	\subsection*{Acknowledgments} Deep thanks to Professor Claudio Muñoz for his constant support, inspiring conversations, kind comments, as well as his advice and valuable suggestions, that have made it possible to obtain a quality version of this work.
	 
	 \section{The initial value problem: local existence}\label{Sect:2}
	
This section is devoted to the proof of Proposition \ref{LOCAL}. First, recall the following result \cite{sogge}, that we will use to prove Proposition \ref{LOCAL}. 
\begin{lem}
Let $\psi: I \times \mathbb{R} \longrightarrow \mathbb{R}$, $I\subseteq \mathbb{R},$ be the solution of the  initial value problem
	\begin{equation}
		\begin{cases*}
			\partial_t^2\psi - \Delta\psi= f(t,x), \quad (t,x) \in I\times \mathbb{R},\\
			(\psi, \partial_t \psi)|_{\{t=0\}}=(\psi_0, \psi_1) \in H^{k}(\mathbb{R})\times H^{k-1}(\mathbb{R}),
		\end{cases*}
	\end{equation}
	where  $k$ be a positive integer. Then for some positive constant $C=C(k),$ the following energy estimate holds   
	\begin{equation}\label{Eenergia}
	\begin{aligned}
		& \sup_{t\in [0,T]} \norm{ (\psi, \partial_t \psi)}_{H^{k}(\mathbb{R})\times H^{k-1}(\mathbb{R})} \leq  C(1+T)\left(  \norm{ (\psi_0,\psi_1) }_{H^{k}(\mathbb{R})\times H^{k-1}(\mathbb{R})}+ \int_{0}^{T} \norm{f}_{H^{k-1}(\mathbb{R})}(t) dt \right).
		\end{aligned}
	\end{equation} 
\end{lem}	 
	 \begin{proof}[Proof of Proposition \ref{LOCAL}] The proof is standard in the literature, but for the sake of completeness, we include it here.
	 
	 \medskip
	 
	 	(1). This part of the Proposition is proved by Picard's iteration. Using a density argument it is sufficient to assume the initial data $(\Psi_0,\Psi_1)\in \mathcal{S}^4$ ($\mathcal S$ being the Schwartz class), along with condition (\refeq{condicion1}). Define a sequence of smooth functions $\Psi^{(i)},$ with $i \geq 1$ such that 
		\[
		\Psi^{(1)}=(0,0),
	 	\]
	 	and for $i \geq 2,$ $\Psi^{(i)}$ is iteratively defined as the unique solution to the system
	 	\begin{equation}
	 		\begin{cases*}
	 			\partial_{\alpha} (m^{\alpha \beta}\partial_{\beta}\Psi^{(i)})=F(\Psi^{(i-1)},\partial \Psi^{(i-1)})\\
	 			(\Psi^{(i)},\partial_t \Psi^{(i)})|_{\{t=0\}}=(\Psi_0, \Psi_1) \in  \mathcal{H}.
	 		\end{cases*}\label{IVP1}
	 	\end{equation}
		 It is important to note that from \eqref{notation} and \eqref{condicion1} we can assure that for $j=1,2,$
	\begin{equation}\label{condicion2}
		\sum_{ \gamma =0, 1} \sup_{|x|,|p|\leq \frac{\lambda}{2}}|\partial^{\gamma}_{x,p} F_j|(x,p) \leq C_{j,\frac12\lambda}.
	\end{equation} 
	Indeed, this can be seen from the fact that for $(x,p)=(x_1,x_2,p_1,p_2,p_3,p_4)$ and $|x|\leq \frac{\lambda}2,$
	\[
	F_1(x,p)= 2\sinh(2\lambda +2x_1)\left(p_4^2-p_3^2\right),\quad F_2(x,p)= \dfrac{\sinh(2(\lambda +x_1))}{\sinh^2(\lambda +x_1)} \left(p_3 p_1 - p_2p_4 \right).
	\]
	Define bounded functions in the class $C^1$.
	
	\medskip
	
	 	It is important to note that condition (\refeq{condicion2}) allows this iterative definition of the functions $\Psi^{(i)}$ to be possible, since it maintains each component of $F$ with the required regularity, see \cite{sogge}. First, it will be shown that for a sufficiently small $T>0,$ the sequence $(\Psi,\partial_t \Psi)$ is uniformly (in $i$) bounded in $L^{\infty}([0,T]; \mathcal{H})$, then it will be shown that it is also a Cauchy sequence. For the first part, the idea is to use the energy estimates (\refeq{Eenergia}), we want to prove that there is a constant $ 0 < A \leq \frac{\lambda}{2}$ such that 
	 	\begin{equation}\label{HI}
	 		\norm{ \left( \Psi^{(i-1)},\partial_t \Psi^{(i-1)} \right) }_{L^{\infty}([0,T];\mathcal{H})} \leq A,	
	 	\end{equation}
	 	implies that 
	 	\begin{equation*}
	 		\norm{ \left(\Psi^{(i)},\partial_t \Psi^{(i)} \right) }_{L^{\infty}([0,T];\mathcal{H})} \leq A.	
	 	\end{equation*} 
	 	The energy estimation (\refeq{Eenergia}) allows us to write for (\refeq{IVP1}) the following estimate:  
	 	\begin{equation}
	 		\begin{aligned}
	 			 \sup_{t\in [0,T]} \norm{\left(\Psi^{(i)},\partial_t \Psi^{(i)}\right)}_{\mathcal{H}} 
	 			 	&~{}\leq  C(1+T)(\norm{\left(\Psi_0, \Psi_1\right)}_{\mathcal{H}})  \\
	 			 +C(1+T)	&~{} \int_{0}^{T} \left(\norm{ F_1\left(\Psi^{(i-1)},\partial \Psi^{(i-1)}\right)}_{L^2(\mathbb{R})}+\norm{ F_2\left( \Psi^{(i-1)},\partial \Psi^{(i-1)} \right)}_{L^2(\mathbb{R})}\right)(t)dt.
	 		\end{aligned} 
	 	\end{equation}
	 	With this estimate, our goal is to bound the integral on the right hand side of the inequality above. For this, we will use the conditions (\refeq{condicion1}) for each $F_j$ which is satisfied by the hypothesis in  (\refeq{HI}), which results in the following, if $B=\max \{C_{1,\frac{\lambda}{2}}, C_{2,\frac{\lambda}{2}}\},$ then 
	 	\begin{equation}
	 		\begin{aligned}
	 			\sup_{t\in [0,T]} \norm{ \left(\Psi^{(i)},\partial_t \Psi^{(i)} \right)}_{\mathcal{H}} \leq  C(1+T)\left( \norm{ \left(\Psi_0, \Psi_1 \right)}_{\mathcal{H}}+2BT \right),
	 		\end{aligned} 
	 	\end{equation}
	 	we can choose $T> 0$  sufficiently small such that 
	 	\begin{equation*}
	 		2BT \leq \norm{\left(\Psi_0,\Psi_1\right)}_{\mathcal{H}},
	 	\end{equation*} 
	 	so
	 	\begin{equation*}
	 	\norm{ \left(\Psi^{(i)},\partial_t \Psi^{(i)}\right)}_{L^{\infty}([0,T]; \mathcal{H})} \leq 2C \norm{ (\Psi_0,\Psi_1) }_{\mathcal{H}}.
	 	\end{equation*}
	 	If we choose $D > 4C$ in (\refeq{condicion1})  and  $A:= 2C||(\Psi_0,\Psi_1)||_{\mathcal{H}}\leq \frac{2C\lambda}{D} \leq \frac{\lambda}{2}.$ We have thus shown the desired implication.
		\\

	 	We now move to the second part in which we show that the sequence is Cauchy in a larger space $L^{\infty}([0,T]; \mathcal{H})$. For every $i \geq 3$ we consider the equation for $\Psi^{(i)}-\Psi^{(i-1)}:$ 
	 	\begin{equation*}
	 		\partial_{\alpha}(m^{\alpha \beta}\partial_{\beta}(\Psi^{(i)}-\Psi^{(i-1)}))= F(\Psi^{(i-1)},\partial \Psi^{(i-1)})-F(\Psi^{(i-2)},\partial \Psi^{(i-2)}).
	 	\end{equation*}
	 	Using the condition (\refeq{HI}) and the mean value theorem to show that there exists some $C>0$ (depending on initial data but
	 	independent of $i$ and $T$) such that
	 	\begin{equation*}
	 		\begin{aligned}
	 			&~{}	 \norm{ F_j\left( \Psi^{(i-1)},\partial \Psi^{(i-1)} \right)-F_j \left(\Psi^{(i-2)},\partial \Psi^{(i-2)} \right)}_{L^2(\mathbb{R})} \leq C \norm{ \partial(\Psi^{(i-1)}- \Psi^{(i-2)}) }_{L^2(\mathbb{R})\times L^2(\mathbb{R})}.
	 		\end{aligned}
	 	\end{equation*}
	 	Now, applying again the energy estimation (\refeq{Eenergia})  
	 	\begin{align*}
	 		&	\sup_{t\in[0,T]} \norm{ (\Psi^{(i)}-\Psi^{(i-1)},\partial_t(\Psi^{(i)}-\Psi^{(i-1)})) }_{\mathcal{H}} \leq  CT \norm{ \left( \Psi^{(i-1)}-\Psi^{(i-2)},\partial_t(\Psi^{(i-1)}-\Psi^{(i-2)}) \right) }_{\mathcal{H}} .
	 	\end{align*}
	 	Using $(\refeq{HI})$ we have 
	 	\begin{equation*}
	 		\sup_{t\in [0,T]} \norm{ \left( \Psi^{(2)}-\Psi^{(1)},\partial_t(\Psi^{(2)}-\Psi^{(1)}) \right)}_{\mathcal{H}} \leq C_1,
	 	\end{equation*}  
	 	therefore, choosing $T$ sufficiently small, for $i \geq 3$ we have 
	 	\begin{equation*}
	 		\sup_{t \in [0,T]} \norm{ \Psi^{(i)}-\Psi^{(i-1)} }_{H^{1}\times H^1} \leq \dfrac{1}{2} \sup_{t \in [0,T]} \norm{ \Psi^{(i-1)}-\Psi^{(i-2)} }_{H^{1}\times H^1},
	 	\end{equation*}
	 	which implies that 
	 	\begin{equation*}
	 		\sup_{t \in [0,T]} \norm{ \Psi^{(i)}-\Psi^{(i-1)} }_{H^{1}\times H^1} \leq \dfrac{C_1}{2^{i-2}}.
	 	\end{equation*}
	 	Therefore we have that the sequence is a Cauchy sequence on $L^{\infty}([0,T]; \mathcal H),$ hence convergent.  That is, there exists $(\Psi, \partial_t \Psi)$ in $L^{\infty}([0,T]; \mathcal H).$ The uniqueness proof, is the result of considering again the energy estimation. 
	 	
	 	Finally, for the continuous dependence on initial data taking $i\in \mathbb{N}$ sufficiently large, let us bound the difference $\Psi^{(i)}-\Psi$ and use again the energy estimate for the equation:
	 		\begin{equation*}
	 		\partial_{\alpha} ({\color{black} m}^{\alpha \beta}\partial_{\beta} (\Psi^{(i)}-\Psi))= F(\Psi^{(i)},\partial \Psi^{(i)})-F(\Psi, \partial \Psi)
	 	\end{equation*}
	 	Applying the same reasoning as above and the energy estimation we can again write:
	 	\begin{align*}
	 		\sup_{s\in [0,t]} \norm{ \left(\Psi^{(i)}-\Psi, \partial_t \Psi^{(i)}-\partial_t \Psi \right) }_{\mathcal{H}} \leq&~{} C \norm{ \left(\Psi_0^{(i)}-\Psi_0,  \Psi_1^{(i)}- \Psi_1 \right)}_{{\color{blue}\mathcal{H}}} \\ 
	 		& +C \int_{0}^{t} \norm{ \left(\Psi^{(i)}-\Psi, \partial_t \Psi^{(i)}-\partial_t \Psi \right)}_{{\color{blue}\mathcal{H}}}.
	 	\end{align*}
	 	{\color{blue} Using Gronwall's inequality (see Appendix \ref{A}) and \eqref{Eenergia}, we have {\color{blue} for a constant $C=C(T) > 0$}},
	 	\begin{equation*}
	 		\sup_{t\in [0,T]} \norm{ \left(\Psi^{(i)}-\Psi, \partial_t \Psi^{(i)}-\partial_t \Psi \right) }_{{\color{blue}\mathcal{H}}} \leq C \norm{ \left(\Psi_0^{(i)}-\Psi_0,  \Psi_1^{(i)}-\Psi_1 \right) }_{{\color{blue}\mathcal{H}}} .
	 	\end{equation*} 
	 Taking $i \longrightarrow \infty$ the right-hand side of the inequality tends to zero, then
	 	\begin{equation*}
	 		\sup_{s\in [0,t]} \norm{ \left(\Psi^{(i)}-\Psi, \partial_t \Psi^{(i)}-\partial_t \Psi \right)}_{{\color{blue}\mathcal{H}}} \longrightarrow 0. 
	 	\end{equation*}	
		This last property ends the proof of Proposition \ref{LOCAL}.
	 \end{proof}
%%%%%%%%%%%%%%%%%%%%%%%%%%%%%%%%%%%%%%%%%%%%%%	 
	\section{Global Solutions for Small Initial Data}\label{Sect:3} 
In this Section we prove Theorem \ref{GLOBAL0}. As in the previous section, let us consider the field $\Lambda(t,x)$ described as $\Lambda(t,x):= \lambda + \tilde{\Lambda}(t,x)$, then, let us establish the conditions on $\lambda$ that guarantee the regularity conditions necessary to study the system (\refeq{cauchy2})
	\begin{equation}
		\begin{cases*}
			\partial^2_{t}\tilde{\Lambda}- \partial_x^2 \tilde{\Lambda}=-2\sinh(2\lambda +2\tilde{\Lambda})\left((\partial_x \phi)^2-(\partial_t\phi)^2 \right) =-2\sinh(2\lambda +2\tilde{\Lambda})Q_0(\phi,\phi),\\
			\partial_t^2 \phi-\partial_x^2 \phi =\dfrac{\sinh(2\lambda +2\tilde{\Lambda})}{\sinh^2(\lambda +\tilde{\Lambda})} \left( \partial_x \phi \partial_x \tilde{\Lambda}-\partial_t \phi \partial_t \tilde{\Lambda} \right)=\dfrac{\sinh(2\lambda +2\tilde{\Lambda})}{\sinh^2(\lambda +\tilde{\Lambda})}Q_0(\phi,\tilde{\Lambda}),\label{problema2}
		\end{cases*}
	\end{equation}
with $Q_0$ given in \eqref{NC}. The constant $\lambda>0$ will play an important role in the overall analysis of the problem and  the conditions assumed on it  will be verified using a continuity method. We will use two coordinate systems: the standard Cartesian coordinates $(t,x)$ and the null coordinates $(u, \underline{u})$: %. The coordinate functions in the null coordinates are the standard optical functions defined as follows:
{\color{blue}
\begin{equation*}
	u:=\dfrac{t-x}{2}, \quad \underline{u}:= \dfrac{t+x}{2}.
\end{equation*}
}
\begin{rem} 
	Consider the two null vector fields defined globally as $$L=\partial_t + \partial_x,\quad \underline{L}= \partial_t - \partial_x.$$  Then, one can rewrite the right-hand side of (\refeq{problema2}) as  
	\begin{align}\label{NC2}
		&	(\partial_x \phi)^2-(\partial_t\phi)^2=Q_0(\phi,\phi)=2L\phi \underline{L}\phi,\\
		& 	\partial_x \phi \partial_x  \tilde{\Lambda}-\partial_t \phi \partial_t \tilde{\Lambda}=Q_0(\phi,\tilde{\Lambda})=\dfrac{1}{2}L\phi \underline{L}\tilde{\Lambda}+\dfrac{1}{2}L\tilde{\Lambda} \underline{L}\phi. \label{NC22}
	\end{align}
	It can be also noticed that the null structure commutes with derivatives: 
	\begin{equation}\label{derivadanullC}
		\partial_xQ_0(\phi,\tilde{\Lambda})=Q_0(\partial_x \phi,\tilde{\Lambda})+Q_0(\phi,\partial_x \tilde{\Lambda}).
	\end{equation}
Also, based on this, we have the following inequality
\begin{equation}\label{importante}
	Q_0(\partial_{x}^p \phi, \partial_{x}^q \phi) \lesssim \abs{L\partial_x^p \phi}\abs{\underline{L}\partial_x^q \phi}+  \abs{\underline{L}\partial_x^p \phi}\abs{L\partial_x^q \phi}.
\end{equation}
\end{rem}
Before presenting the proof, there are certain results and definitions to be mentioned before, for details and proofs see \cite{alinhac2009hyperbolic, Luli2018}. From now on, we will consider  the conformal killing vector field on $\mathbb{R}^{1+1}$ given by
\begin{equation*}
	(1+\abs{{\color{blue}\underline{u}}}^2)^{1+\delta}L, \quad (1+\abs{{\color{blue}u}}^2)^{1+\delta}\underline{L},
\end{equation*}
 with $0 < \delta < 1$, and the following integration regions: ${\color{blue}\Sigma}_{t_0}$  denotes the following time slice in $\mathbb{R}^{1+1}$: 
\begin{align}\label{S_t}
	{\color{blue}\Sigma}_{t_0} := \{(t,x):\; t=t_0 \}.
\end{align}
$D_{t_0}$ denotes the following region of spacetime 
\begin{equation}\label{D_t}
	D_{t_o}:= \{ (t,x):\;  0 \leq t \leq t_0 \}, \quad D_{t_0}=\bigcup_{0\leq t \leq t_0} {\color{blue}\Sigma}_{t_0}.
\end{equation}
The level sets of the functions $u$ and $\underline{u}$ define two global null foliations of $D_{t_0}$. More precisely, given $t_0>0$, $u_0$ and $\underline{u}_0$, we define the rightward null curve segment {\color{blue}$C_{u_0}$} as :
{\color{blue}\begin{equation*}
	C_{u_0} := \left\{(t,x): \; u=\frac{t-x}{2} =u_0,\, 0\leq t \leq t_0\right\},
\end{equation*}}
and the segment of the null curve to the left {\color{blue}$\underline{C}_{\underline{u}_0}$}as:
{\color{blue}\begin{equation*}
	\underline{C}_{\underline{u}_0} := \left\{(t,x): \; \underline{u}=\frac{t+x}{2} =\underline{u}_0,\, 0\leq t \leq t_0\right\}.
\end{equation*}}

\begin{figure}
\centering
\begin{tikzpicture}[scale=1.0]
\def\aa{7} 	
\begin{axis}[
axis x line = middle,
axis y line = left,
every outer y axis line/.style={draw=none}, 
axis line style = {-},
% no y axis
%y axis line style = {draw=none},
%xlabel = $$,
%ylabel = $$,
xmin = -11,
xmax = 18,
clip=false, % es para que permita salirse de los valores máximo y mínimo dados de ser necesari
xtick = {-11},
xticklabels ={},
ymin = 0,
ymax = 7,
ytick={0,7},
yticklabels={$\Sigma_0$,$\Sigma_t$},
height = 10em,
width = 20em,
]

\addplot+[
black,%,very thick,
mark=none,
%const plot,
%empty line=jump,
]
coordinates {
(-7,7)
(0,0)
%(7,0)
%(13,7)
}
node[rotate = 0,pos=0.45,left]{$\underline{C}_{\underline u}$};
% horizontal part
\addplot+[
black,%,very thick,
mark=none,
%const plot,
%empty line=jump,
]
coordinates {
%(-7,7)
%(0,0)
(7,0)
(13,7)
}
node[rotate = 0,pos=.55,right]{$C_u$};

\addplot+[
black,%,very thick,
mark=none,
%const plot,
%empty line=jump,
]
coordinates {
(-11,7)
(18,7)
};
\end{axis}
\end{tikzpicture}
\caption{{\color{blue} The entire region enclosed by $\Sigma_0$ and $\Sigma_t$ is $D_t.$}} \label{fig R}
\end{figure}
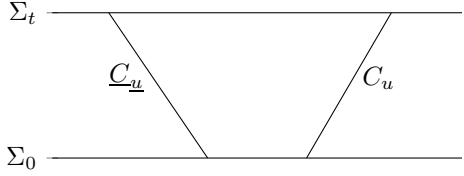

The space time region $D_{t_0}$ is foliated by {\color{blue}$\underline{C}_{\underline{u}_0}$} for $\underline{u}\in \mathbb{R}$, and by {\color{blue}$C_{u_0}$} for $u\in \mathbb{R}$. 
Let us also consider the following  energy estimate  proposed in \cite{alinhac2009hyperbolic, Luli2018} for the scalar linear equations $\square \psi = \rho$ given by:
\begin{equation}
	\begin{aligned}\label{EEnergia}
		& ~{}\int_{{\color{blue}\Sigma_t}} \left[(1+\abs{{\color{blue}u} }^2)^{1+\delta}\abs{\underline{L}\psi}^2+ (1+\abs{{\color{blue}\underline{u}}}^2)^{1+\delta}\abs{L\psi}^2\right]dx \\
		& \qquad +\sup_{{\color{blue}\underline{u}}\in \mathbb{R}}\int_{\underline C_{{\color{blue}\underline{u}}}} (1+\abs{{\color{blue} u}}^2)^{1+\delta}\abs{\underline{L}\psi}^2d\tau + \sup_{{\color{blue} u}\in \mathbb{R}}\int_{C_{{\color{blue} u}}} (1+\abs{{\color{blue}\underline u}}^2)^{1+\delta}\abs{L\psi}^2 d\tau \\
		&~{} \qquad \leq C_0\int_{{\color{blue}\Sigma}_0}\left[ (1+\abs{{\color{blue} u}}^2)^{1+\delta}\abs{\underline{L}\psi}^2+(1+\abs{{\color{blue} \underline u}}^2)^{1+\delta}\abs{L\psi}^2 \right]dx\\
		& ~{} \qquad \quad + C_0 \iint_{D_t} \left[(1+\abs{{\color{blue}u}}^2)^{1+\delta}\abs{\underline{L}\psi} + (1+\abs{{\color{blue}\underline u}}^2)^{1+\delta}\abs{L\psi}\right]\abs{\rho}.
	\end{aligned}
\end{equation}
Motivated by the above estimation (\refeq{EEnergia}) and \cite{Luli2018}, we define the space-time weighted energy norms valid for $k=0,1$:
\begin{equation}\label{energies}
\begin{aligned}
	&	\mathcal{E}_k(t)=\int_{{\color{blue}\Sigma}_t} \left[(1+\left|{\color{blue} u} \right|^2)^{1+\delta}\abs{\underline{L}\partial_x^k\tilde{\Lambda}}^2+  (1+\left|{\color{blue} \underline u} \right|^2)^{1+\delta}\abs{L\partial_x^k \tilde{\Lambda}}^2\right]dx,\\
	&	\overline{\mathcal{E}}_k(t)=\int_{{\color{blue}\Sigma}_t} \left[(1+\left|{\color{blue} u} \right|^2)^{1+\delta}\abs{\underline{L}\partial_x^k\phi}^2+  (1+\left|{\color{blue} \underline u} \right|^2)^{1+\delta}\abs{L\partial_x^k \phi}^2\right]dx,\\
	& \mathcal{F}_k(t)= \sup_{{\color{blue} \underline u}\in \mathbb{R}} \int_{{\color{blue} \underline{C}_{\underline u}}} (1+\left|{\color{blue} u}\right|^2)^{1+\delta} \left| \underline{L}\partial_x^k\tilde{\Lambda} \right|^2ds+ \sup_{{\color{blue} u}\in \mathbb{R}} \int_{{\color{blue} C_{u}}} (1+\left| {\color{blue}\underline u} \right|^2)^{1+\delta}\abs{L\partial_x^k \tilde{\Lambda}}^2ds,\\
	& \overline{\mathcal{F}}_k(t)= \sup_{{\color{blue}\underline u}\in \mathbb{R}} \int_{C_{{\color{blue} \underline u}}} (1+\left|{\color{blue} u} \right|^2)^{1+\delta} \left| \underline{L}\partial_x^k\phi \right|^2ds+ \sup_{{\color{blue} u}\in \mathbb{R}} \int_{C_{{\color{blue} u}}} (1+\left|{\color{blue}\underline u} \right|^2)^{1+\delta}\abs{L\partial_x^k \phi}^2ds.
\end{aligned}
\end{equation}
Finally, we define the total energy norms as follows:
\[
\mathcal{E}(t)= \mathcal{E}_0 (t)+\mathcal{E}_1(t).
\] 
Analogously one defines $\mathcal{F}(t)$, $\overline{\mathcal{E}}(t)$, and $\overline{\mathcal{F}}(t).$ 
\begin{rem}
	We note that if $t=0$ then $\mathcal{F}(0)=\overline{\mathcal{F}}(0)=0$ and for $\mathcal{E}(t)$ the initial data determines a constant $C_1$ so that
	\begin{equation}
		\mathcal{E}(0)=C_1 \varepsilon^2.
	\end{equation}
	\end{rem}
	We will use the method of continuity as follows: we assume that the solution $\tilde{\Lambda}$ exists for $t\in [0,T^{*}]$ so that it has the following bound
	\begin{align}\label{supuesto}
		&	\mathcal{E}(t)+\mathcal{F}(t) \leq 6C_0C_1\varepsilon^2,\\
		&	\overline{\mathcal{E}}(t)+ \overline{\mathcal{F}}(t) \leq 6C_0\overline{C}_1\varepsilon^2\label{supuesto1},
	\end{align}
and
\begin{equation}\label{condicionlambda}
	\sup_{t \in [0,T^{*}]}\norm{\tilde{\Lambda}}_{L^{\infty}(\mathbb{R})} \leq \dfrac{\lambda}{2}.
\end{equation}
	We want to show that for all $t\in [0,T^*]$ there exists a universal constant $\varepsilon_0$ (independent of $T^{*}$) such that the estimates are improved for all $\varepsilon \leq \varepsilon_0$. 
It is important recall that, the terms related to the functions $\sinh(\cdot)$, $\cosh(\cdot)$  $\coth(\cdot)$ and $\csch(\cdot)$ can be written using the Taylor expansion as:
\begin{equation}
	\begin{cases}
		\sinh(2\lambda +2\tilde{\Lambda})= \sinh(2\lambda)+2\cosh(2\lambda)\tilde{\Lambda}+ 4 \sinh(2\lambda +2\xi_1)\tilde{\Lambda}^2,\\
		\cosh(2\lambda +2\tilde{\Lambda})= \cosh(2\lambda)+2\sinh(2\lambda)\tilde{\Lambda} +4 \cosh(2\lambda +2\xi_2) \tilde{\Lambda}^2,\\
	\end{cases}\label{expansion}
\end{equation}
and {\color{blue} under \eqref{condicionlambda},}
{\color{blue}\begin{equation}
\begin{cases}
&\coth(\lambda+\tilde \Lambda)=\coth(\lambda)-\csch(\lambda)\tilde \Lambda - \csch(\lambda +\xi_3)\coth(\lambda +\xi_3)\tilde \Lambda^2,\\
&\csch^2(\lambda + \tilde \Lambda)=\csch^2(\lambda) -2 \csch^2(\lambda)\coth(\lambda)\tilde \Lambda \\
&\quad \quad \quad \quad \quad \quad \quad  +\left\{ 2\csch^2(\lambda +\xi_4)\coth^2(\lambda +\xi_4)+\csch^4(\lambda + \xi_4) \right\}\tilde{\Lambda}^2,
\end{cases}\label{expansion1}
\end{equation}}
with $\xi_1,\xi_2, \xi_3, \xi_4$ between $0$ and $\tilde{\Lambda}$, which satisfies \eqref{condicionlambda}. Then, from this condition \eqref{condicionlambda} and (\refeq{expansion}) one has 
{\color{blue}\begin{equation}\label{cotas}
\begin{aligned}
	&\abs{\sinh(2\lambda+ 2\tilde{\Lambda})}\leq \lambda_0(\lambda), \quad
	\abs{\cosh(2\lambda+2\tilde{\Lambda})}\leq \lambda_1(\lambda).\\
	& \abs{\coth(\lambda +\tilde \Lambda)} \leq \lambda_3(\lambda), \quad  \hbox{and} \quad  \abs{\csch(\lambda+\tilde \Lambda )}\leq \lambda_4(\lambda).
	\end{aligned}
\end{equation}}
Using the assumptions (\refeq{supuesto}) and (\refeq{supuesto1}) the following pointwise bounds were established in \cite{Luli2018}. 
\begin{lem}[\cite{Luli2018}, Lemma 3.2]\label{lema1}
	Under assumptions  \eqref{supuesto}-\eqref{condicionlambda}, there exists a universal constant  $C_2>0$ such that:
	\begin{align*}
		&	|L\tilde{\Lambda}(t,x)|\leq \dfrac{C_2 \varepsilon}{(1+|{\color{blue}\underline u} |^2)^{1/2+\delta/2}}, &  	|L\phi(t,x)|\leq \dfrac{C_2 \varepsilon}{(1+|{\color{blue} \underline u} |^2)^{1/2+\delta/2}},\\
		& 	|\underline{L}\tilde{\Lambda}(t,x)|\leq \dfrac{C_2 \varepsilon}{(1+|{\color{blue}u} |^2)^{1/2+\delta/2}}, &  	|\underline{L}\phi(t,x)|\leq \dfrac{C_2 \varepsilon}{(1+|{\color{blue}u} |^2)^{1/2+\delta/2}}.
	\end{align*} 	
\end{lem}

{\color{blue}\begin{proof}
It is sufficient to prove one of the four inequalities, since the other inequalities are completely analogous. The proof is based on the boostrap assumptions  \eqref{supuesto}-\eqref{condicionlambda}. Indeed, according to the Sobolev inequality on $\mathbb{R}$, since 
\[\abs{\partial_x (1+\abs{\underline{u}}^2)^{1/2+\delta/2}} \leq  (1+\abs{\underline{u}}^2)^{1/2+\delta /2},\]
 one has 
 \begin{equation*}
 \begin{aligned}
 \abs{(1+\abs{\underline{u}}^2)^{1/2+\delta /2} L \phi (t,x)}^2 &\lesssim || (1+\abs{\underline{u}}^2)^{1/2+\delta /2} L\phi||_{L^2(\mathbb{R}_x)}^2  +   ||\partial_x( (1+\abs{\underline{u}}^2)^{1/2+\delta /2} L\phi)||_{L^2(\mathbb{R}_x)}^2\\
 & \lesssim   || (1+\abs{\underline{u}}^2)^{1/2+\delta /2} L\phi||_{L^2(\mathbb{R}_x)}^2  + ||\partial_x( (1+\abs{\underline{u}}^2)^{1/2+\delta /2}) L\phi||_{L^2(\mathbb{R}_x)}^2\\
    & \quad   + || (1+\abs{\underline{u}}^2)^{1/2+\delta /2} L\partial_x( \phi)||_{L^2(\mathbb{R}_x)}^2\\
 & \lesssim   || (1+\abs{\underline{u}}^2)^{1/2+\delta /2} L\phi||_{L^2(\mathbb{R}_x)}^2  + || (1+\abs{\underline{u}}^2)^{1/2+\delta /2} L\partial_x( \phi)||_{L^2(\mathbb{R}_x)}^2\\
& \lesssim \overline{\mathcal{E}}(t)\\
& \lesssim 6C_0\overline{C}_1 \varepsilon^2. 
 \end{aligned}
 \end{equation*}
Consequently, we have the desired inequality. 
\end{proof}}
Now we have all the ingredients to prove Theorem \ref{GLOBAL0}.
\subsection{Proof of the Theorem \ref{GLOBAL0}}
For simplicity, we work with the first equation of the system (\refeq{problema3}). On the other hand, we can substitute in the following proof $\tilde{\Lambda}$ by $\phi$ and then sum the estimates to complete the test for the original system, {\color{blue} see Appendix \ref{B} for details of the estimates for the second equation in \eqref{problema2}, which complete the proof.} We prove this using the bootstrap method; i.e., we will assume that this weighted energy is bounded by some constant. Then, we can show that the solution decays. Since the initial data are small, this allows us to show that the weighted energy is bounded by some better constant.  Thus, by continuity, we conclude that the weighted energy cannot grow to infinity in any finite time interval and therefore, using the local existence theorem, the solution exists for all time.

\begin{proof} Using (\refeq{NC2}) and (\refeq{derivadanullC}) in the first equation of the (\refeq{problema2}) we obtain:
	\begin{equation}\label{derivadas}
		 \square \partial_x \tilde{\Lambda}  = -2\Big[\sinh(2\lambda +2\tilde{\Lambda})\left(Q_0(\partial_x\phi,\phi)+Q_0(\phi,\partial_x\phi)\right)+2\partial_x\tilde{\Lambda}\cosh(2\lambda+2\tilde{\Lambda})Q_0(\phi,\phi)\Big].
	\end{equation}
We can see that the null structure is ``quasi-preserved'' after differentiating with respect to $x$. We will use a bootstrap argument as in the (3+1)-dimensional case \cite{klainerman}. Fix $\delta \in (0,1)$. Under the assumptions (\refeq{supuesto})-(\refeq{supuesto1})-(\refeq{condicionlambda}) for all $t\in [0,T^{*}]$, we assume that the solution remains regular, to later show that these bounds are maintained, with a better constant. 

Consider $k=0,1$. Using \eqref{EEnergia} on \eqref{derivadas}, with $\psi=\partial^k_x \tilde{\Lambda}$, and taking the sum over $k=0,1$, we obtain
\begin{equation}\label{EE2}
	\begin{aligned}
		&\mathcal{E}(t)+\mathcal{F}(t) \leq 2C_0 \mathcal{E}(0)\\
		&\quad +2C_0  \iint_{D_t} \left( (1+\left|u \right|^2)^{1+\delta} | \underline{L}\tilde{\Lambda}|+ (1+\left|\underline{u} \right|^2)^{1+\delta}|L \tilde{\Lambda}|\right) |\sinh(2\lambda +2\tilde{\Lambda})| |Q_0(\phi,\phi)| \\
		&   \quad          +4C_0\iint_{D_t} \left( (1+\left|u \right|^2)^{1+\delta} | \underline{L}\partial_x\tilde{\Lambda}|+ (1+\left|\underline u \right|^2)^{1+\delta}| L\partial_x \tilde{\Lambda}|\right) |\sinh(2\lambda+2\tilde{\Lambda})||Q_0(\phi,\partial_x\phi)| \\
				& \quad +4C_0 \iint_{D_t} \left( (1+\left|u \right|^2)^{1+\delta} | \underline{L}\partial_x\tilde{\Lambda}|+ (1+\left|\underline u \right|^2)^{1+\delta}| L\partial_x \tilde{\Lambda}|\right) |\partial_x\tilde{\Lambda} \cosh(2\lambda +2\tilde{\Lambda})||Q_0(\phi,\phi)| \\
		& =: 	2C_0 \mathcal{E}(0)+ 2C_0\sum_{j=1}^6 I_j,
	\end{aligned}
\end{equation}
{\color{blue}where the integrals $I_i, i\in \{ 1,2,3...,6\}$ are defined as follows:
\begin{equation}
\begin{aligned}
&I_1:=    \iint_{D_t} \left( (1+\left|u \right|^2)^{1+\delta} | \underline{L}\tilde{\Lambda}|\right) |\sinh(2\lambda +2\tilde{\Lambda})| |Q_0(\phi,\phi)|,\\
& I_2:=  \iint_{D_t}\left(  (1+\left| \underline u \right|^2)^{1+\delta}|L \tilde{\Lambda}|\right) |\sinh(2\lambda +2\tilde{\Lambda})| |Q_0(\phi,\phi)|,\\
& I_3:= 2 \iint_{D_t} \left( (1+\left|u \right|^2)^{1+\delta} | \underline{L}\partial_x\tilde{\Lambda}|\right) |\sinh(2\lambda+2\tilde{\Lambda})||Q_0(\phi,\partial_x\phi)|, \\
& I_4:= 2 \iint_{D_t} \left( (1+\left| \underline u \right|^2)^{1+\delta}| L\partial_x \tilde{\Lambda}|\right) |\sinh(2\lambda+2\tilde{\Lambda})||Q_0(\phi,\partial_x\phi)| ,\\
& I_5:= 2 \iint_{D_t} \left( (1+\left|u \right|^2)^{1+\delta} | \underline{L}\partial_x\tilde{\Lambda}|\right) |\partial_x\tilde{\Lambda}\cosh(2\lambda +2\tilde{\Lambda})||Q_0(\phi,\phi)|, \\
& I_6:=2\int_{D_t}\left( (1+\left| \underline u \right|^2)^{1+\delta}| L\partial_x \tilde{\Lambda}|\right) |\partial_x\tilde{\Lambda} \cosh(2\lambda +2\tilde{\Lambda})||Q_0(\phi,\phi)|.
\end{aligned}
\end{equation}}
	 The goal is to control the right-hand side of the above estimate. Essentially we have six terms to control, but several are equivalent and we only need to consider essentially two cases.  Indeed, it will be sufficient to bound the terms corresponding to $\underline{L}\tilde{\Lambda}$ and $\underline{L}\partial_x\tilde{\Lambda}$, since by symmetry, the procedure for the other terms will be analogous. 
First, we start to bound the term:
\begin{equation}\label{integral1}
\begin{aligned}
	I_{35} := I_3 +I_5= &~{} \iint_{D_t} \left( (1+\left|u \right|^2)^{1+\delta} \abs{ \underline{L}\partial_x\tilde{\Lambda}}\right) \\
	&~{} \qquad  \left(2  |\partial_x\tilde{\Lambda} \cosh(2\lambda +\tilde{\Lambda})||Q_0(\phi,\phi)|+2|\sinh(2\lambda+2\tilde{\Lambda})||Q_0(\phi,\partial_x\phi)|\right).
	\end{aligned}
\end{equation}
 Taking into account (\refeq{importante}), (\refeq{condicionlambda}) and (\refeq{expansion}), we can write for (\refeq{integral1}):
\begin{equation}\label{integral2}
\begin{aligned}
	I_{35} \lesssim &~{}  \iint _{D_t} \left( (1+\left|u \right|^2)^{1+\delta}| \underline{L}\partial_x\tilde{\Lambda}|\right)\left( |\partial_x\tilde{\Lambda}||L\phi|| \underline{L}\phi|+|L\partial_x\phi ||\underline{L}\phi|+ |L\phi| |\underline{L}\partial_x\phi|\right) \\
	=: &~{} I_{35,1}+I_{35,2} +I_{35,3}.
	\end{aligned}
\end{equation} 
Since $\partial_x\tilde{\Lambda}=\frac{1}{2}(L-\underline{L})\tilde{\Lambda}$, we have
\begin{equation}\label{I351}
\begin{aligned}
I_{35,1}= &~{} \iint _{D_t} (1+\left|u\right|^2)^{1+\delta}|\underline{L}\partial_x \tilde{\Lambda}||\partial_x\tilde{\Lambda}||L\phi||\underline{L}\phi| \\
 \leq &~{} \frac12  \iint_{D_t}(1+\left|u \right|^2)^{1+\delta}|\underline{L}\partial_x\tilde{\Lambda}||L\tilde{\Lambda}||L\phi||\underline{L}\phi|+(1+\left| \underline{u} \right|^2)^{1+\delta}|\underline{L}\partial_x \tilde{\Lambda}||\underline{L}\tilde{\Lambda}||L\phi||\underline{L}\phi| \\
 =: &~{} I_{35,1,1} +I_{35,1,2}.
\end{aligned}
\end{equation}
Recall that by Fubini's Theorem the spacetime $D_t$ in \eqref{D_t} is foliated by $\underline{C}_{\underline u}$ for $u \in \mathbb{R}$, and also by $\{t\}\times \Sigma_t$, $t\in\mathbb R$. Using Lemma \refeq{lema1} and defining $\varphi(x)=(1+\left|x \right|^2)^{1+\delta}$ (to simplify the notation), we have the following:
\begin{align*}
	  I_{35,1,1} \lesssim & \iint_{D_t} \varepsilon \underbrace{(\varphi(\underline u)^{-3/4}\varphi(u)^{1/2}|\underline{L}\partial_x \tilde{\Lambda}|)}_{L_t^2L_x^2}\underbrace{(\varphi^{1/2}(\underline u)|L\phi|)}_{L_t^{\infty}L_x^2}\underbrace{(\varphi(\underline u)^{-1/4}\varphi (u)^{1/2}|\underline{L}\phi|)}_{L_t^2 L_x^{\infty}}\\
	& \lesssim \varepsilon \underbrace{\left(\iint_{D_t} \dfrac{\varphi (u)|\underline{L}\partial_x \tilde{\Lambda}|^2}{\varphi(\underline u)^{3/2}} \right)^{1/2}}_{T_1} 
	\underbrace{\sup_{t\in [0,T^*]}\left(\int_{\Sigma_t}\varphi(\underline u)|L \phi|^2 \right)^{1/2}}_{T_2} 
	\underbrace{\left(  \int_{0}^{t} \norm{ \dfrac{\varphi (u)^{1/2}}{\varphi(\underline u)^{1/4}} |\underline{L}\phi|}_{L^{\infty}({\color{blue}\Sigma}_{\tau})}^{2} d\tau  \right)^{1/2}.}_{T_3}	
\end{align*}
Let us study each of the factors $T_j$. For $T_1$, one has:
\begin{align*}
	T_1^2 \leq & \int_{\mathbb{R}}\left[\int_{\underline{C}_{\underline u}} \dfrac{\varphi (u)|\underline{L}\partial_x \tilde{\Lambda}|^2}{\varphi(\underline u)^{3/2}} ds\right]d\underline{u} = \int_{\mathbb{R}} \dfrac{1}{\varphi  (\underline u)^{3/2}}\underbrace{\left[\int_{\underline{C}_{\underline u}} \varphi (u)|\underline{L}\partial_x \tilde{\Lambda}|^2ds \right]}_{\lesssim \mathcal{F}_1(t)} d\underline{u}
	\lesssim  \int_{\mathbb{R}} \dfrac{\varepsilon^2}{\varphi  (\underline u)^{3/2}}d\underline{u},
\end{align*} 
since the integral is finite, we have
$T_1 \lesssim \varepsilon.$
 The integral $T_2$ is part of the energy norm  $\overline{\mathcal{E}}_0(t)$ in \eqref{energies} then $T_2 \lesssim \varepsilon.$
For the term $T_3$ one can use the same argument as in \cite{Luli2018}: using Lemma \refeq{lemaaux} one gets
{\color{blue}\begin{align*}
	T_3 &\lesssim  \left(\int_0^t  \norm{\frac{\varphi(u)^{1/2}}{\varphi(\underline u)^{1/4}} \underline{L}\phi(t,x) }^2_{L^{2}(\Sigma_{\tau})} +\int_0^t  \norm{\frac{\varphi(u)^{1/2}}{\varphi(\underline u)^{1/4}} \underline{L}\partial_x\phi(t,x) }_{L^{2}(\Sigma_{\tau})}^2  \right)^{1/2} \\
	&\lesssim \left( \iint_{D_t} \frac{\varphi(u)}{\varphi(\underline u)^{1/2}}|\underline{L}\phi|^2 + \iint_{D_t} \frac{\varphi(u)}{\varphi(\underline u)^{1/2}}|\underline{L}\partial_x \phi|^2  \right)^{1/2}.
\end{align*}}
Both terms above are of the same form as $T_1$ and then we have that $T_3 \lesssim \varepsilon.$ We conclude  that $  I_{35,1,1} \lesssim  \varepsilon^4.$ 
 
 \medskip
 
Now we control the integral $I_{35,1,2}$ in \eqref{I351}. Using again Lemma \refeq{lema1} we have:
\begin{align*}
	I_{35,1,2} = &  \iint_{D_t}(1+\left|u \right|^2)^{1+\delta}|\underline{L}\partial_x \tilde{\Lambda}||\underline{L}\tilde{\Lambda}||L\phi||\underline{L}\phi|   \lesssim \iint_{D_t} \varepsilon^2 |\underline{L}\partial_x \tilde{\Lambda}||L\phi|\\
	=& \iint_{D_t} \varepsilon^2 \dfrac{\varphi (u)^{1/2}}{\varphi(\underline u)^{1/2}}|\underline{L}\partial_x\tilde{\Lambda}|\dfrac{\varphi(\underline u)^{1/2}}{\varphi (u)^{1/2}}|L\phi|\\
	\lesssim & \iint_{D_t} \varepsilon^2 \left( \dfrac{\varphi (u)}{\varphi(\underline u)}|\underline{L}\partial_x\tilde{\Lambda}|^2 + \dfrac{\varphi(\underline u)}{\varphi (u)}|L\phi|^2 \right)\\
	\lesssim & \int_{\mathbb{R}} \dfrac{\varepsilon^2}{\varphi(\underline u)}\underbrace{\left[ \int_{\underline{C}_{\underline u}} \varphi (u) |\underline{L}\partial_x \tilde{\Lambda}|^2ds \right]}_{\lesssim \mathcal{F}_1(t)} d\underline{u} + \int_{\mathbb{R}} \dfrac{\varepsilon^2}{\varphi(u)}\underbrace{\left[ \int_{C_{u}} \varphi (\underline u) |L\phi|^2ds \right]}_{\lesssim \overline{\mathcal{F}}_0 (t)} du
	\lesssim \varepsilon^4.
\end{align*}
Putting all estimates together for $I_{35,1},$ we can conclude that $I_{35,1} \lesssim \varepsilon^4.$ A similar result is obtained for $I_{46}:= I_4+I_6$. %Now let us consider a second case:\\ \newline

\medskip

Now we treat the term $I_1+I_2+I_{35,2}+I_{35,3}$ from \eqref{EE2} and \eqref{integral2}. We have from \eqref{importante} and \eqref{cotas},
\[
 \iint _{D_t} \varphi(u ) | \underline{L}\partial_x\tilde{\Lambda}| \left(  |L\partial_x\phi ||\underline{L}\phi|+ |L\phi| |\underline{L}\partial_x\phi|\right)+\iint_{D_t} \left( \varphi(u)| \underline{L}\tilde{\Lambda}|+ \varphi(\underline u)|L \tilde{\Lambda}|\right)\left( |L\phi ||\underline{L}\phi|+ |L\phi| |\underline{L}\phi|\right)
\]
Using the condition (\refeq{condicionlambda}), the situation matches Case 1 developed in \cite{Luli2018}.  All these integrals can be written as
\begin{align*}
	\sim  \iint_{D_t}  \left(\varphi (u) |\underline{L}\partial_x \tilde{\Lambda}||L\phi||\underline{L}\partial_x \phi| + \varphi (u) |\underline{L}\partial_x \tilde{\Lambda}||L\partial_x \phi| |\underline{L}\phi|\right).
\end{align*}
We bound this term in the following form: take $j,k\in\{0,1\}$, $\psi=\tilde \Lambda,\phi$, so that
\begin{equation}\label{bounds}
\begin{aligned}
	\iint_{D_t}\varphi(u)\left| \underline{L}\partial_x^{k}\psi\right|\left| L\psi \right|\left|\underline{L}\partial_x^j \psi\right| & \lesssim \iint_{D_t}\frac{\varepsilon}{\varphi (\underline u)^{1/2}}\varphi(u)\left| \underline{L}\partial_x^k \psi\right| \left| \underline{L}\partial_x^j\psi\right|\\
	  & \lesssim  \iint_{D_t}\frac{\varepsilon}{\varphi (\underline u)^{1/2}}\left(\varphi(u)\left|\underline{L}\partial_x^k \psi\right|^2+\varphi(u)\left|\underline{L}\partial_x^j \psi\right|^2\right)\\
	  & \lesssim \int_{\mathbb{R}}\left[ \int_{\underline{C}_{\underline u}}\frac{\varepsilon}{\varphi(\underline u)^{1/2}}\left(\varphi(u)\left|\underline{L}\partial_x^k \psi\right|^2+\varphi(u)\left|\underline{L}\partial_x^j \psi\right|^2\right)ds \right]d\underline{u}\\
	  &= \int_{\mathbb{R}}\frac{\varepsilon}{\varphi(\underline u)^{1/2}} \underbrace{\left[\int_{\underline{C}_{\underline{u}}}\left(\varphi(u)\left|\underline{L}\partial_x^k \psi\right|^2+\varphi(u)\left|\underline{L}\partial_x^j \psi\right|^2\right)ds\right]}_{\lesssim \mathcal E +\mathcal F + \overline{\mathcal{E}}+ \overline{\mathcal{F}}} d\underline{u}\\
	  & \lesssim \int_{\mathbb{R}} \frac{\varepsilon^3}{\varphi(\underline u)^{1/2}}d\underline{u} \lesssim \varepsilon^{3}.
\end{aligned}
\end{equation}
See also Luli, Yan and Yu \cite{Luli2018} for detailed computations. So we can conclude that in this case  we can bound them by $\varepsilon^3$. Finally, from  the energy estimate (\refeq{EEnergia}), we can take all the estimates together for some universal constant $C_4, C_5$ we have that for all $t\in [0, T^{*}]$:
\begin{equation}
	\mathcal{E}(t)+\mathcal{F}(t) \leq  2C_0C_1\varepsilon^2 + C_4 \varepsilon^3 + C_5 \varepsilon^4.
\end{equation}
Now, we take $\varepsilon_0 $ such that
\begin{equation}\label{condicionfinal}
	\varepsilon_0 \leq \dfrac{C_0 C_1}{C_4}, \quad  \varepsilon_0^2 \leq \dfrac{C_0 C_1}{C_5},
\end{equation}
we can see that for all $0 < \varepsilon \leq \varepsilon_0$ and for all $t\in [0, T]$, we have
\begin{equation}
	\mathcal{E}(t)+\mathcal{F}(t) \leq 4C_0C_1 \varepsilon^2.
\end{equation}
This improves the constant in (\refeq{supuesto})
In the same way, an analogous reasoning is used for the analysis of the equation in terms of $\phi$, using in this case the equation (\refeq{NC22}), which results in an improvement of the constant involved in the estimate (\refeq{supuesto1}). 

\medskip

To improve condition (\refeq{condicionlambda}),  using the Fundamental Theorem of Calculus and Lemma \refeq{lema1}, one can write $\tilde{\Lambda}(t,x)$, $t\geq 0$, in the following form:
\begin{align*}
	\abs{\tilde{\Lambda}(t,x)}&\leq \varepsilon \abs{\tilde{\Lambda}_0(x)} +\int_{0}^t \abs{\partial_t\tilde{\Lambda}(\tau,x)}d\tau\\
	&\leq \varepsilon K_1 +\dfrac{1}{2}\int_0^t \abs{L\tilde{\Lambda}+\underline{L}\tilde{\Lambda}}d\tau\\
	& \leq \varepsilon K_1+ \dfrac{1}{2}\int_0^t \left(\dfrac{C_2 \varepsilon}{\varphi(\underline u)^{1/2}} +\dfrac{C_2 \varepsilon}{\varphi(u)^{1/2}}\right)d\tau\\
	& \leq \varepsilon K_1 + \varepsilon C_2K_2\leq K \varepsilon,
\end{align*}
for some universal constant $K.$ Next, we take $\varepsilon_0 >0$ that satisfies the condition $(\refeq{condicionfinal})$ and such that
\begin{equation}\label{condfinal2}
	K\varepsilon_0 < \dfrac{\lambda}{4},
\end{equation}
taking $\sup$ over $t\in [0,T^{*}]$, we conclude that for all $0 < \varepsilon \leq \varepsilon_0$ we improved estimate (\refeq{condicionlambda}). As mentioned before, the proof is completed by doing an analogous study in terms of the $\phi$ field and then taking the sum over the estimates for the final conclusion.

\end{proof}
%%%%%%%%%%%%%%%%%%%%%%%%%%%%%%%%%%%%%%%%%%%%%	

	\section{Long time behavior}\label{Sect:4}

	Recall the energy introduced in \eqref{Energy}: 
	\begin{equation*}
		E[\Lambda,\phi](t) = \int
		\left( \dfrac{1}{2}((\partial_{x}\Lambda)^{2}+(\partial_{t}\Lambda)^{2})+2\sinh^{2}(\Lambda)((\partial_{x}\phi)^{2}+(\partial_{t}\phi)^{2}) \right)(t,x)dx.
	\end{equation*}
	We first start with a simple computation, already present in \cite{yaronhadad_2013}.
	\begin{lem}\label{conservation}
		If $\Lambda(t,x), \phi(t,x)$ are the solutions  of \eqref{eq:Chiral_0} with $\Lambda(t,x)\in C^{\infty}_0(\mathbb{R}) $ and $\phi(x)\in C^{\infty}_0(\mathbb{R})$ then the energy of the system is conserved, that is
		\begin{equation*}
			\frac{d}{dt}E[\Lambda,\phi](t)=0.
		\end{equation*}
	\end{lem} 
%For a rigorous proof of this result, see Appendix \ref{A}.
\subsection{Energy and momentum densities} In terms of the fields $\Lambda$ and $\phi$, let us introduce the energy and momentum densities
	\begin{equation}\label{CE}
		\begin{aligned}
			&~{} p(t,x):= \partial_x\Lambda\partial_t\Lambda + 4\sinh^{2}(\Lambda)\partial_x\phi\partial_t\phi,\\
			&~{} e(t,x):=  \dfrac{1}{2}((\partial_x\Lambda)^{2}+(\partial_t\Lambda)^{2}) + 2\sinh^{2}(\Lambda)((\partial_x\phi)^{2}+(\partial_t\phi)^{2}).
		\end{aligned}
	\end{equation}
	\begin{lem}\label{densities}
		 Using the definition above in Eq. \eqref{CE}, one has the following continuity equations
		\begin{equation}\label{eq:continuidad}
			\begin{aligned}	
				\partial_{t}p(t,x) =&~{} \partial_{x} e(t,x), \\
				\partial_{t}e(t,x) =&~{} \partial_{x} p(t,x),
			\end{aligned}
		\end{equation}
		and the inequality
		\begin{equation}\label{desigualdad}
		|p(t,x)| \leq e(t,x).
		\end{equation}
	\end{lem}
	
	\begin{proof}
		
		First we prove $\partial_{t}p(t,x) = \partial_{x} e(t,x)$. Using \eqref{sistema4} we can prove the continuity  equation \eqref{eq:continuidad}.
		Let us star with the first derivatives
		\begin{align*}
			&\partial_{x}\left(-\dfrac{1}{2}((\partial_x\Lambda)^{2}+(\partial_t\Lambda)^{2})-2\sinh^{2}(\Lambda)((\partial_x\phi)^{2}+(\partial_t\phi)^{2})\right)\\
			&=  -\partial_x\Lambda \partial_x^2\Lambda+\partial_t\Lambda\partial_{tx}\Lambda-2\sinh(2\Lambda)((\partial_x\phi)^{2}+(\partial_t\phi)^{2})\partial_x\Lambda  -4\sinh^{2}(\Lambda)(\partial_x\phi \partial_x^2\phi+\partial_t\phi\partial_{tx}\phi),
		\end{align*}
		and	
		\begin{align*}	
			& \partial_{t}\left(\partial_x\Lambda\partial_t\Lambda  + 4\sinh^{2}(\Lambda)\partial_x\phi\partial_t\phi \right)\\ 
			& = \partial_{xt}\Lambda\partial_t\Lambda+\partial_x\Lambda\partial_t^2\Lambda+4\sinh(2\Lambda)\partial_t\Lambda\partial_x\phi \partial_t\phi  +4\sinh^{2}(\Lambda)\partial_{xt}\phi\partial_t\phi+4\sinh^{2}(\Lambda)\partial_z\phi \partial_{t}^2\phi.
		\end{align*}
		Subtracting these two last equations gives:
		\begin{align*}
			& 2\sinh{2\Lambda}((\partial_x\phi)^{2}+(\partial_t\phi)^{2})\partial_x\Lambda+\partial_x\Lambda(\partial_x^2\Lambda-\partial_t^2\Lambda)\\
			&+4\sinh^{2}(\partial_x^2\phi-\partial_t^2\phi)\partial_x\phi-4\sinh(2\Lambda)\partial_t\Lambda \partial_x\phi\partial_t\phi\\ & = 4\sinh(2\Lambda)(\partial_x\phi)^{2}\partial_x\Lambda-4(\partial_x\phi)^{2}\partial_x\Lambda\sinh(2\Lambda)+4\sinh(2\Lambda)\partial_x\phi\partial_t\phi\partial_t\Lambda -4\sinh(2\Lambda)\partial_t\Lambda\partial_x\phi\partial_t\phi=0.
		\end{align*}
		
Second, we prove $\partial_{t}e(t,x) = \partial_{x} p(t,x)$; in effect, using \eqref{sistema4}  we have
		\begin{equation*}
			\begin{aligned}
				\partial_t e(t,x)=&\partial_x\Lambda \partial_{xt}\Lambda + \partial_t \Lambda \partial_{tt}\Lambda  +2 \partial_t\Lambda\sinh(2\Lambda)((\partial_x\phi)^2+(\partial_t\phi)^2)\\
				& +4\sinh^2(\Lambda)\partial_x\phi \partial_{xt}\phi+4\sinh^2(\Lambda)\partial_t\phi \partial_{tt}\phi\\
				=& 	-2\partial_t\Lambda \sinh(2\Lambda)((\partial_x\phi)^2-(\partial_t\phi)^2)+2 \partial_t\Lambda\sinh(2\Lambda)((\partial_x\phi)^2+(\partial_t\phi)^2)\\
				&+4\sinh^2(\Lambda)\partial_x\phi \partial_{xt}\phi +4\partial_t\phi \partial_{xx}\phi\sinh^2(\Lambda)+4\sinh(2\Lambda)\partial_t\phi \partial_x\phi\partial_t\Lambda\\
				&+\partial_x \Lambda\partial_{xt}\Lambda+\partial_t\Lambda \partial_{xx}\Lambda-4(\partial_t\phi)^2\partial_t\Lambda\sinh(2\Lambda)\\
				& =\partial_x(\partial_x\Lambda \partial_t\Lambda +4\sinh^2(\Lambda)\partial_t\phi \partial_x\phi).
			\end{aligned}
		\end{equation*}	
		Then, the equation \eqref{eq:continuidad} is satisfied. As we can see, the continuity equation can be written explicitly as
		\begin{equation}
			\begin{aligned}
		&~{}	\partial_{t} \left( \partial_{x}\Lambda\partial_{t}\Lambda+4\sinh^{2}{\Lambda}\partial_{x}\phi\partial_{t}\phi \right)\\
			&~{}\quad \quad \quad \quad-\partial_{x}\left( \frac{1}{2}((\partial_{x}\Lambda)^{2}+(\partial_{t}\Lambda)^{2}) + 2\sinh^{2}(\Lambda)((\partial_{x}\phi)^{2}+(\partial_{t}\phi)^{2}) \right)=0.
			\end{aligned}
		\end{equation}
	To prove the inequality, let us take into account Cauchy's  inequality, then 
	\begin{equation*}
		\begin{aligned}
			&~{} |\partial_x\Lambda \partial_t\Lambda| \leq \frac12 \left( (\partial_x\Lambda)^2 + (\partial_t\Lambda)^2\right) ,\qquad	|\partial_x\phi \partial_t\phi|  \leq  \frac12 \left( (\partial_x\phi)^2+(\partial_t\phi)^2\right),
		\end{aligned}
	\end{equation*}
so that
\begin{equation}\label{desgEM}
	|p(t,x)|\leq \dfrac{1}{2}((\partial_x\Lambda)^{2}+(\partial_t\Lambda)^{2}) + 2\sinh^{2}(\Lambda)((\partial_x\phi)^{2}+(\partial_t\phi)^{2}).
\end{equation}
That is, the energy density exerts a control on the momentum density, which will be of key importance, since all the analysis and results will attempt to establish the energy space of the coupled system.
	\end{proof}
%	\subsection{Cauchy theory}
	\subsection{Virial estimate}
	The purpose of this section is to present a Virial identity
	which is related to the energy  presented above. Let us take into account certain considerations  following a proposal similar to the one used in \cite{Alejo2018}. However, in our case the semilinear character of the model enters and no smallness in a smaller space is needed. In what follows, we consider  $t \geq 2$ only, and
	\begin{equation}\label{lambda}
		\omega (t):=\dfrac{t}{\log ^2 t}, \quad \quad \frac{\omega'(t)}{\omega(t)}=\frac{1}{t}\left(1-\dfrac{2}{\log t}\right).
	\end{equation}
Furthermore, let us consider $(\Lambda,\phi)$ continuous in time such that $E[\Lambda,\phi](t)<+\infty$ is conserved. We introduce a  Virial identity for the chiral field equation \eqref{sistema4}. Indeed, let $\rho:= \tanh(\cdot)$, and let $\mathcal{I}(t)$ be defined as
	\begin{equation}\label{virial}
		\mathcal{I}(t):=-\int_{\mathbb R} \rho\left(\frac{x-vt}{\omega(t)}\right) \left( \partial_x \Lambda\partial_t \Lambda+4\partial_x \phi \partial_t \phi\sinh^2(\Lambda)\right)dx, \quad v\in (-1,1).
	\end{equation}
	A time-dependent weight was also
	considered in \cite{Alejo2018},  with the same goals. The choice of $\mathcal{I}(t)$ is  motivated by the momentum and energy densities. Recall that $\int = \int_{\mathbb R}.$
\begin{lem}[Virial identity]\label{Virial2} We have
	\begin{equation}\label{virial1}
		\begin{aligned}
			\frac{d}{dt}\mathcal{I}(t)=	&~{}  \dfrac{\omega'(t)}{\omega(t)}\int \frac{x-vt}{\omega(t)} \rho'\left(\frac{x-vt}{\omega(t)}\right)(\partial_x \Lambda\partial_t \Lambda+4\partial_x \phi \partial_t \phi\sinh^2(\Lambda))
			\\
			&~{} 	 +\frac{1}{\omega(t)}\int \rho'\left(\frac{x-vt}{\omega(t)}\right)\left(\frac{1}{2} (\partial_x\Lambda)^2+2(\partial_t\phi)^2\sinh^2(\Lambda) \right)	\\
			&~{} 	+\frac{1}{\omega(t)}\int\rho'\left(\frac{x-vt}{\omega(t)}\right) \left( \frac{1}{2}(\partial_t\Lambda)^2+2(\partial_x\phi)^2\sinh^2(\Lambda) \right)\\
			&~{} 	+ \frac{v}{\omega(t)}\int\rho'\left(\frac{x-vt}{\omega(t)}\right) \left( \partial_x \Lambda\partial_t \Lambda+4\partial_x \phi \partial_t \phi\sinh^2(\Lambda) \right).
		\end{aligned}
	\end{equation}
\end{lem}	

	\begin{proof} From \eqref{eq:continuidad} we readily have
		\begin{equation*}
			\begin{aligned}
			\dfrac{d }{d t}\mathcal{I}(t)= &~{} \dfrac{\omega'(t)}{\omega(t)}\int \rho' \left(\frac{x-vt}{\omega(t)}\right)\frac{x-vt}{\omega(t)}p(t,x)	 + \frac{v}{\omega(t)} \int \rho' \left(\frac{x-vt}{\omega(t)}\right) p(t,x)\\
			&~{} -\int \rho \left(\frac{x-vt}{\omega(t)}\right)\partial_x e(t,x),
			\end{aligned}
		\end{equation*}	
	using integration by parts and the Lemma \ref{densities}
	\begin{equation*}
		\begin{aligned}
			\frac{d}{dt}\mathcal{I}(t)= &~{} \dfrac{\omega'(t)}{\omega(t)}\int \rho' \left(\frac{x-vt}{\omega(t)}\right)\frac{x-vt}{\omega(t)}p(t,x) + \frac{v}{\omega(t)} \int \rho' \left(\frac{x-vt}{\omega(t)}\right) p(t,x)\\
			&~{}  +\frac{1}{\omega(t)}\int \rho' \left(\frac{x-vt}{\omega(t)}\right) e(t,x).  %\dfrac{\lambda'(t)}{\lambda(t)}\int \varphi' \left(\frac{x}{\lambda(t)}\right)\frac{x}{\lambda(t)}p(t,x)	+\frac{1}{\lambda(t)}\int \varphi' \left(\frac{x}{\lambda(t)}\right) e(t,x)
		\end{aligned}
	\end{equation*}
This proves \eqref{virial1} after replacing \eqref{CE}.
	\end{proof}
	
\subsection{Integration of the dynamics} 
The goal of this subsection is prove the Theorem \ref{LTD}, let us start with the following integral estimate
\begin{lem}\label{estimacion}
	Let $\omega(t)$ given as in \eqref{lambda}. Assume that the solution $(\Lambda, \phi)(t)$ of the system \eqref{sistema4} satisfies 
	\begin{equation}\label{condition}
	E[\Lambda, \phi](t)<+\infty.
\end{equation}
 Then  we have the averaged estimate
	\begin{equation}\label{desg1}
		\int_2^{\infty} \frac{1}{\omega(t)}\int_{\mathbb R}\sech^2\left(\frac{x-vt}{\omega(t)}\right)e(t,x)dxdt \lesssim  1,
	\end{equation}
	Moreover, there exists an increasing sequence $t_n\to +\infty$ such that 
	\begin{equation}\label{desg2}		\lim_{n\longrightarrow +\infty} \int_{\mathbb R} \sech^{2}\left(\frac{x-vt_n}{\omega(t_n)}\right)e(t_n,x)dx=0.
	\end{equation}
\end{lem}
In order to show Lemma \eqref{estimacion}, we use the new Virial identity for \eqref{virial} presented
for the Chiral Field Equation  \eqref{eq:Chiral_0}.
\begin{proof}
	First	note that, from the condition \eqref{condition} we have that clearly $\mathcal{I}(t)$ in \eqref{virial} is well defined. Here we use the fact that
	both $\partial_x\Lambda$ and $\partial_t\Lambda$ are small in $L^{\infty}$  thanks to the Sobolev embedding and in of the same form $\partial_x\phi$ and $\partial_t\phi$. Moreover 
	\begin{equation}\label{condition2}
		\sup_{t\in \mathbb{R}}|\mathcal{I}(t)|\lesssim E[\Lambda,\phi](t) \lesssim 1.
	\end{equation}
On the other hand, from Lemma \ref{Virial2}, we have the identity 
	\begin{equation}
		\frac{d}{dt}\mathcal{I}(t)=\mathcal{J}_1+\mathcal{J}_2,
	\end{equation}
	where 
	\begin{equation*}
		\begin{aligned}
			\mathcal{J}_1=\dfrac{\omega'(t)}{\omega(t)}\int \rho' \left(\frac{x-vt}{\omega(t)}\right)  \frac{x-vt}{\omega(t)} \left( \partial_x \Lambda\partial_t \Lambda+4\partial_x \phi \partial_t \phi\sinh^2(\Lambda) \right),
		\end{aligned}
	\end{equation*}
	and $\mathcal{J}_2$ is the remaining term of  \eqref{virial1}. From the definition of $\omega(t)$, \eqref{condition} and using Cauchy's inequality for $\delta>0$ small, we have:
	\begin{equation*}
		\begin{aligned}
			|\mathcal{J}_1|&~{}\leq \frac{2}{t} \int \frac{|x-vt|}{\omega(t)} \rho'\left(\frac{x-vt}{\omega(t)}\right)(|\partial_x \Lambda||\partial_t\Lambda|+4|\partial_x\phi||\partial_t\phi|\sinh^2(\Lambda))\\
			&~{}\leq \frac{8C_\delta}{t^2}\int\frac{(x-vt)^2}{\omega(t)}\rho'\left(\frac{x-vt}{\omega(t)}\right)\left(\frac{1}{2}(\partial_t\Lambda)^2+2(\partial_t\phi)^2\sinh^2(\Lambda)\right) \\
			&~{} \quad +\frac{\delta}{\omega(t)}\int\rho'\left(\frac{x-vt}{\omega(t)}\right)\left(\frac{1}{2}(\partial_x\Lambda)^2+2(\partial_x\phi)^2\sinh^2(\Lambda)\right)\\
			&~{}\leq \frac{8C_\delta\omega(t)}{t^2}\sup_{x\in \mathbb{R}}\left(\frac{(x-vt)^2}{\omega^2(t)}\rho' \left(\frac{x-vt}{\omega(t)}\right)\right)\int\left(\frac{1}{2}(\partial_t\Lambda)^2+2(\partial_t\phi)^2\sinh^2(\Lambda)\right)\\
			&~{}\quad +\frac{\delta}{\omega(t)}\int \rho'\left(\frac{x-vt}{\omega(t)}\right)\left(\frac{1}{2}(\partial_x\Lambda)^2+2 (\partial_t\phi)^2\sinh^2(\Lambda)\right)\\
			&~{}\leq \frac{C}{t\log^2t}+\frac{\delta}{\omega(t)}\int \rho'\left(\frac{x-vt}{\omega(t)}\right)\left(\frac{1}{2}(\partial_x\Lambda)^2+2(\partial_t\phi)^2\sinh^2(\Lambda)\right).
		\end{aligned}
	\end{equation*} 
	Furthermore, for $\mathcal J_2(t)$ we have
	\[
	\begin{aligned}
	\frac{|v|}{\omega(t)}\int\rho'\left(\frac{x-vt}{\omega(t)}\right) \left| \partial_x \Lambda\partial_t \Lambda+4\partial_x \phi \partial_t \phi\sinh^2(\Lambda) \right| = &~{} \frac{|v|}{\omega(t)}\int\rho'\left(\frac{x-vt}{\omega(t)}\right) |p(t,x)|\\
	\leq  &~{} \frac{|v|}{\omega(t)}\int\rho'\left(\frac{x-vt}{\omega(t)}\right) e(t,x).
	\end{aligned}
	\]
	With this estimate on $\mathcal{J}_1$ to obtain $1-|v|-\delta>0$ for $\delta>0$ sufficiently small, and
	\begin{equation}\label{rate_decay_1}
		\frac{d}{dt}\mathcal{I}(t)\geq \dfrac{1-|v| -\delta}{\omega(t)}\int \rho'\left(\frac{x-vt}{\omega(t)}\right)e(t,x)-\frac{C}{t\log^2t}.
	\end{equation}
After integration in time we get \eqref{desg1}. Finally, \eqref{desg2} is obtained from \eqref{desg1} and
the fact that $\omega^{-1}(t)$ is not integrable in $[2,\infty)$.
\end{proof}
	\begin{proof}[Proof of Theorem \ref{LTD}]
		Let us consider $\psi(\cdot)=(\rho^{\prime})^2=\sech^4(\cdot),$ then
\begin{equation*}
\begin{aligned}
	\dfrac{d}{dt}\int \psi\left( \frac{x-vt}{\omega(t)}\right)e(t,x)= &~{} -\frac{\omega^{\prime}(t)}{\omega(t)}\int \frac{x-vt}{\omega(t)}\psi^{\prime}\left(\frac{x-vt}{\omega(t)}\right)e(t,x) -\frac{v}{\omega(t)}\int \psi' \left( \frac{x-vt}{\omega(t)}\right)e(t,x) \\
	&~{} + \frac{1}{\omega(t)}\int \psi^{\prime} \left(\frac{x-vt}{\omega(t)}\right)p(t,x).
	\end{aligned}
\end{equation*}		
Since $\abs{\frac{x-vt}{\omega(t)}\psi^{\prime}\left( \frac{x-vt}{\omega(t)}\right) }\lesssim \sech^2\left(\frac{x-vt}{\omega(t)}\right)$ and $|p(t,x)|\leq e(t,x)$	we have:
	\begin{equation}\label{eq:Aux}
		\abs{\dfrac{d}{dt}\int \psi\left( \frac{x-vt}{\omega(t)}\right)e(t,x)}\leq \frac{C}{\omega(t)}\int \sech^2\left(\frac{x-vt}{\omega(t)}\right)e(t,x),
	\end{equation}
furthermore
\begin{equation}\label{eq:auxiliar}
	\lim_{n\longrightarrow \infty} \int \sech^{4}\left(\frac{x-vt_n}{\omega(t_n)}\right)e(t_n,x)=0. %Aquí estas bajo las hipotesis del virial, es decir, t_n es una sucesión creciente...ojo con eso!!!
	\end{equation}
Finally using (\refeq{eq:Aux}) for $t < t_n$
\begin{equation*}
\abs{\int \psi\left( \frac{x-vt_n}{\omega(t_n)}\right)e(t_n,x)-\int \psi\left( \frac{x-vt}{\omega(t)}\right)e(t,x)} \leq \int_t^{t_n} \frac{2}{\omega(s)}\int \sech^2\left( \frac{x-vs}{\omega(s)}\right)e(s,x)dxds,
\end{equation*}
sending $n$ to infinity, and using (\refeq{eq:auxiliar})  we have
\begin{equation}\label{rate_decay0}
	\abs{\int \psi\left( \frac{x-vt}{\omega(t)}\right)e(t,x)}\leq \int_{t}^{\infty} \frac{2}{\omega(s)}\int \sech^2\left( \frac{x-vs}{\omega(s)}\right)e(s,x)dxds,
\end{equation}
which implies, thanks to Lemma \refeq{estimacion},
\begin{equation*}
	\lim_{t \longrightarrow \infty} \int \sech^4\left( \frac{x-vt}{\omega(t)}\right)e(t,x) =0,
\end{equation*} 
which finally shows the validity of Theorem \refeq{LTD}. 
%{\color{blue} In order to show \eqref{rate_decay}, from \eqref{rate_decay_1} and \eqref{virial} we get
%\[
%\int_t^{t_n} \left(   \dfrac{1}{\omega(t)}\int \rho'\left(\frac{x-vt}{\omega(t)}\right)e(t,x) \right) dt  \leq C \left( \mathcal{I}(t_n) -\mathcal I(t) \right) + \frac{C}{\log t};
%\]
%and from \eqref{rate_decay0} we have for some $C>0$ independent of time,
%\[
%\begin{aligned}
%\abs{\int \psi\left( \frac{x-vt}{\omega(t)}\right)e(t,x)}\leq &~{}  \int_{t}^{\infty} \frac{2}{\omega(s)}\int \sech^2\left( \frac{x-vs}{\omega(s)}\right)e(s,x)dxds \\
%\leq &~{}  C \left( \mathcal{I}(t_n) -\mathcal I(t) \right) + \frac{C}{\log t} +\int_{t_n}^{\infty} \frac{2}{\omega(s)}\int \sech^2\left( \frac{x-vs}{\omega(s)}\right)e(s,x)dxds .
%\end{aligned}
%\]
%In particular, given any $\varepsilon>0$, if $n$ is large enough, 
%\[
%\abs{\int \sech^4 \left( \frac{x-vt}{\omega(t)}\right)e(t_n,x)}\leq  \frac{C}{\log t_n} +\varepsilon. % \int_{t_n}^{\infty} \frac{2}{\omega(s)}\int \sech^2\left( \frac{x-vs}{\omega(s)}\right)e(s,x)dxds.
%\]
%This estimate proves \eqref{rate_decay}.
%}
\end{proof}

\section{Application to soliton solutions}\label{Sect:5}

In this section, we apply our previous results to prove existence of global solutions around a new class of soliton solutions of finite energy. First, we consider the case treated by Hadad in \cite{yaronhadad_2013}. See also \cite{letelier1986, economou1989, letelier1985static} for other cases of soliton-like solutions not treated here.

\subsection{Singular solitons}  Consider the soliton introduced in \eqref{soliton}. We claim that this solution is singular in the narrow sense that the energy of the system for this soliton is not finite.  Our first result is the following straightforward computation:
\begin{lem}\label{LP_soliton}
One has, 
\begin{equation}\label{LP_new}
\begin{aligned}
	\Lambda(t,x)= & \ln(|v|\cosh(t)) \\
	& +\ln\left(1-\frac{\tanh(t)\tanh(\sqrt{c}(x-v t))}{|v|\sqrt{c}}+\sqrt{\left(1-\frac{\tanh(t)\tanh(\sqrt{c}(x-v t))}{|v|\sqrt{c}}\right)^2-\frac{\sech^2(t)}{|v|^2}}\right),\\
	\phi(t,x)=& \frac\pi4 -\frac12\arctan \left[ \cosh(t)\cosh(\sqrt{c}(x-v t))(\tanh(\sqrt{c}(x-v t))+v\sqrt{c}\tanh(t))\right].
\end{aligned}
\end{equation}
Moreover, for $E_\text{mod}$ given in \eqref{Energia_simpleM}, the previous solution gives
\begin{equation}\label{Energy_mod_LP}
	E_\text{mod}[\Lambda,\phi](t) =0.
\end{equation}
\end{lem}

\begin{rem}
Notice that $g^{(0)}$ in \eqref{g0} has also zero modified energy. This is in concordance with the fact that $g^{(1)}$ is obtained from $g^{(0)}$ as seed. 
\end{rem}

\begin{proof}	
We use the notation in \eqref{soliton} and $\gamma:=\sqrt{c}(x-v t)$. Comparing the soliton (\refeq{soliton}) with (\refeq{diag1}) we have the following equations:
\begin{eqnarray}
	 e^t[\cosh(\ln \mu)-\sinh(\ln \mu)\tanh(\gamma)] &= & \cosh(\Lambda)+\cos(2\phi)\sinh(\Lambda),\\
	  e^{-t}[\cosh(\ln \mu)+\sinh(\ln \mu)\tanh(\gamma)]& =& \cosh(\Lambda)-\cos(2\phi)\sinh(\Lambda),\\
	 -\dfrac{1}{\sqrt{c}\cosh(\gamma)} &= & \sin(2\phi)\sinh(\Lambda),\label{eq:ultima}
\end{eqnarray}
where
\begin{eqnarray*}
\cosh(\ln(\mu))= \dfrac{\mu^2+1}{2\mu}=-v,&
	\sinh(\ln (\mu))=\dfrac{\mu^2-1}{2\mu}=\dfrac{1}{\sqrt{c}},
\end{eqnarray*}
 adding  the first two equations we obtain:
\[
	 -v\cosh(t)-\frac{1}{\sqrt{c}}\sinh(t)\tanh(\gamma)=\cosh \Lambda.
\]
Then, since we have the constraint $\mu >1$ we can  write the expression for $\Lambda$ as
\begin{align*}
	\Lambda(t,x)=
	%& \cosh^{-1}\left( -v\cosh(t)-\frac{1}{\sqrt{c}}\sinh(t)\tanh(\sqrt{c}(x-v t))\right)\\
	&\ln(|v|\cosh(t)) \\
	& +\ln\left(1-\frac{\tanh(t)\tanh(\gamma)}{|v|\sqrt{c}}+\sqrt{\left(1-\frac{\tanh(t)\tanh(\gamma)}{|v|\sqrt{c}}\right)^2-\frac{1}{|v|^2}\sech^2(t)}\right).
\end{align*}
Next, subtracting the same equations and using (\refeq{eq:ultima})
\begin{eqnarray*}
	&-v\sinh(t)-\frac{1}{\sqrt{c}}\cosh(t)\tanh(\gamma)=\cos(2\phi)\sinh(\Lambda),\\
	& -v\sinh(t)-\frac{1}{\sqrt{c}}\cosh(t)\tanh(\gamma)=-\frac{1}{\sqrt{c}}\cot(2\phi)\sech(\gamma)\\
	&  \sinh(\gamma)\cosh(t) + \sqrt{c}v\sinh(t)\cosh(\gamma)=\cot(2\phi),
\end{eqnarray*}
In order to make sense, one needs $\sin(2\phi)\neq 0,$ i.e., $\phi \neq \dfrac{n \pi}{2}$. Therefore, we can write:
\begin{align*}
	\phi(t,x)%=&\frac{1}{2}\cot^{-1}\left(\sinh(\gamma)\cosh(t)-\omega\sinh(t)\cosh(\gamma)\right)\\
	 =&\frac\pi4 -\frac12\arctan \left(\sinh(\gamma)\cosh(t) + \sqrt{c}v\sinh(t)\cosh(\sqrt{c}(x-v t))\right)\\
	 =&{\color{black} \frac\pi4 -\frac12\arctan \left( \cosh(t)\cosh(\sqrt{c}(x-v t))(\tanh(\sqrt{c}(x-v t))+v\sqrt{c}\tanh(t))\right).}
\end{align*}
Now, let us study the derivatives of the $\Lambda$ and $\phi$ fields. Assuming the constraints for the parameter $\mu$ we have that 
\begin{equation*}
	\begin{aligned}
	&\partial_x \Lambda= -\dfrac{\sinh(t)\sech^2(\gamma)}{\sqrt{(-v\cosh(t)-\frac{1}{\sqrt{c}}\sinh(t)\tanh(\gamma))^2-1}},\\
	&\partial_t \Lambda = \frac{\tanh(\gamma)\left(-v\sinh(t)\tanh(\gamma)-\frac{1}{\sqrt{c}}\cosh(t)\right)}{\sqrt{(-v\cosh(t)-\frac{1}{\sqrt{c}}\sinh(t)\tanh(\gamma))^2-1}}.
	\end{aligned}
	\end{equation*}
Additionally,
\begin{equation*}
	\begin{aligned}
		&\partial_x \phi=-\frac{1}{2}\left(\dfrac{\sqrt{c}\cosh(t)\cosh(\gamma)(1+v\sqrt{c}\tanh(t)\tanh(\gamma))}{1+(\sinh(\gamma)\cosh(t)+v\sqrt{c}\sinh(t)\cosh(\gamma))^2}\right), \\
		&\partial_t\phi=-\frac12\left( \dfrac{(1-cv^2)\sinh(t)\sinh(\gamma)}{1+(\sinh(\gamma)\cosh(t)+v\sqrt{c}\sinh(t)\cosh(\gamma))^2}\right).
	\end{aligned}
\end{equation*}
Simplifying, the energy density is:
\begin{equation*}
\begin{aligned}
& (\partial_x \Lambda)^2+(\partial_t \Lambda)^2-1\\
& =\frac{\sinh^2(t)\sech^4(\gamma)-v^2\sinh^2(t)\sech^2(\gamma)(\tanh^2(\gamma)+1)-\frac{v}{\sqrt{c}}\sinh(2t)\tanh(\gamma)\sech^2(\gamma)-v^2+1}{(-v\cosh(t)-\frac{1}{\sqrt{c}}\sinh(t)\tanh(\gamma))^2-1},
\end{aligned}	
\end{equation*}
and 
\begin{equation*}
\begin{aligned}
& \sinh^2(\Lambda)((\partial_x \phi)^2+(\partial_t \phi)^2) \\
& = \frac{c\cosh^2(t)\cosh^2(\gamma)-\frac{vc}{2}\sinh(2t)\sinh(2\gamma)+(c^2v^4-cv^2+1)\sinh^2(t)\sinh^2(\gamma)}{(1+(\sinh(\gamma)\cosh(t)+v\sqrt{c}\sinh(t)\cosh(\gamma))^2)^2}.
\end{aligned}
\end{equation*}
Then, the integrals can be calculated with the help of the computer algebra system Mathematica, obtaining that the soliton has finite modified energy, in fact, we have that
\[
E_\text{mod}[\Lambda,\phi](t)=0.
\]
With the results obtained we can see that the $\Lambda$ and $\phi$ fields associated to the soliton (\refeq{soliton}) do not belong to the energy space proposed in the previous sections. 
\end{proof}

\subsection{Finite energy solitons} In this final section we consider the case of finite energy solitons, their perturbations, and a corresponding global well-posedness result.

\begin{proof}[Proof of Corollary \ref{aplication}:]
	Identifying the 1-soliton in \eqref{solitonG}
	\begin{equation}\label{solitonG0}
		g^{(1)}=	\left[\begin{array}{cc}
			\dfrac{e^{\lambda +\varepsilon\theta}\sech(\beta(\lambda + \varepsilon \theta))}{\sech(\beta(\lambda + \varepsilon \theta)-x_0)} & -\dfrac{1}{\sqrt{c}}\sech(\beta(\lambda + \varepsilon \theta))\\
			-\dfrac{1}{\sqrt{c}}\sech(\beta(\lambda + \varepsilon \theta))& \dfrac{e^{-(\lambda + \varepsilon \theta)}\sech(\beta(\lambda + \varepsilon \theta))}{\sech(\beta(\lambda + \varepsilon \theta)+x_0)}
		\end{array} \right], \qquad \beta=\frac{\mu +1}{\mu -1},
	\end{equation}
with the geometrical representation (\refeq{diag1}), one gets the corresponding fields $\hat{\Lambda}_{\varepsilon}$ and $\hat{\phi}_{\varepsilon}$, which have the following form:
\begin{equation*}
\begin{aligned}
	\hat{\Lambda}_\varepsilon(t,x):= &~{} \cosh^{-1}\left(|v|\cosh(\lambda+\varepsilon \theta)-\frac{1}{\sqrt{c}}\tanh(\beta(\lambda+\varepsilon\theta))\sinh(\lambda+\varepsilon \theta)\right),\\
	\hat{\phi}_\varepsilon(t,x):=  &~{}  \frac{\pi}{4}-\dfrac{1}{2}\arctan\left(\cosh(\beta(\lambda+\varepsilon\theta))\cosh(\lambda+\varepsilon \theta)(\tanh(\beta(\lambda+\varepsilon \theta))+v\sqrt{c}\tanh(\lambda+\varepsilon \theta))\right),
\end{aligned}	
\end{equation*}
which are solutions for (\refeq{problema3}). From now on we drop $\varepsilon$ to make the notation less cumbersome. 

\medskip

We claim that $\hat{\Lambda}$  have the desired local and global well-posedness properties. Indeed, note that since  $0< \mu<1$, then $|v|>1$ and $\beta <0,$ so, for all $t,x\in \mathbb{R}$
\begin{equation*}
\begin{aligned}
&	|v|\cosh(\lambda+\varepsilon \theta)-\frac{1}{\sqrt{c}}\tanh(\beta(\lambda+\varepsilon\theta))\sinh(\lambda+\varepsilon \theta)\\
& \qquad  \geq |v|+\frac{1}{\sqrt{c}}\tanh(|\beta|(\lambda+\varepsilon\theta))\sinh(\lambda+\varepsilon \theta) > 1,
\end{aligned}
\end{equation*}
therefore, $\hat{\Lambda}$ is well-defined and $\hat{\Lambda}(t,x) > 0 $ for all $t,x\in \mathbb{R}$. Also, since $\theta \in L^{\infty}(\mathbb{R})$, for each $t$, we have to that $\hat{\Lambda}$ is a bounded function. Since $\theta\in C_0^{2},$ we have that 
\begin{equation*}
%\lim\limits_{x\longrightarrow \pm\infty}
\hat{\Lambda}(t=0,x)=C(\lambda), \quad \forall x\in \mathbb{R}\setminus \supp \theta,
\end{equation*}
then, define $\tilde{\lambda}:= C(\lambda)$, which allows us to write $\hat{\Lambda}:= \tilde{\Lambda}+ \tilde{\lambda}$. For the  function $\tilde{\Lambda}$ one has
\begin{eqnarray}
	\tilde{\Lambda}|_{\{t=0\}}= \varepsilon\tilde{\Lambda}_0, \quad \mbox{with}\quad  \tilde{\Lambda}_0\in C_0^{2}(\mathbb{R}).
	\end{eqnarray}
where $\tilde{\Lambda}_0$ is defined as:
% With the definition of $\tilde{\lambda}$ we have
%\begin{align*} \tilde{\Lambda}(t=0,x)&=\hat{\Lambda}(t=0,x)-\cosh^{-1}\left(|v|\cosh(\lambda)-\frac{1}{\sqrt{c}}\tanh(\beta \lambda)\sinh(\lambda)\right)\\
%&=  \cosh^{-1}\left(|v|\cosh(\lambda+\varepsilon \theta)-\frac{1}{\sqrt{c}}\tanh(\beta( \lambda+\varepsilon \theta))\sinh(\lambda+\varepsilon \theta)\right)\\
% &\qquad -\cosh^{-1}\left(|v|\cosh(\lambda)-\frac{1}{\sqrt{c}}\tanh(\beta \lambda)\sinh(\lambda)\right) 
%\end{align*}
 \begin{align*}
	\tilde{\Lambda}_0(x)&:=\frac{1}{\varepsilon}\left(\cosh^{-1}\left(|v|\cosh(\lambda+\varepsilon \theta(x))-\frac{1}{\sqrt{c}}\tanh(\beta( \lambda+\varepsilon \theta(x)))\sinh(\lambda+\varepsilon \theta(x))\right) \right)\\
	&\quad -\frac{1}{\varepsilon}\left(\cosh^{-1}\left(|v|\cosh(\lambda)-\frac{1}{\sqrt{c}}\tanh(\beta \lambda)\sinh(\lambda)\right) \right).
\end{align*}
 The dependence associated with $\varepsilon$ for this function, is suitable in the sense that we can demonstrate straightforwardly that $\tilde{\Lambda}_0$ is a bounded function when $\varepsilon$ tends to zero, indeed, we have that the $\lim\limits_{\varepsilon \longrightarrow 0} \tilde{\Lambda}_0$ can be calculated using  L'H\^opital's rule:
\begin{align*}
	\lim\limits_{\varepsilon \longrightarrow 0} \tilde{\Lambda}_0&= \lim\limits_{\varepsilon \longrightarrow 0} \frac{\theta(x)\left(|v|\sinh(\lambda +\varepsilon\theta)-\frac{\beta}{\sqrt{c}}\sech^2(\beta(\lambda +\varepsilon \theta))-\frac{1}{\sqrt{c}}\tanh(\beta(\lambda+\varepsilon \theta))\cosh(\lambda+\varepsilon \theta)\right)}{\sqrt{(|v|\cosh(\lambda+\varepsilon \theta)-\frac{1}{\sqrt{c}}\tanh(\beta( \lambda+\varepsilon \theta))\sinh(\lambda+\varepsilon \theta))^2-1}}\\
	&  =\frac{\theta(x)\left(|v|\sinh(\lambda )-\frac{\beta}{\sqrt{c}}\sech^2(\beta\lambda)-\frac{1}{\sqrt{c}}\tanh(\beta\lambda)\cosh(\lambda)\right)}{\sqrt{(|v|\cosh(\lambda)-\frac{1}{\sqrt{c}}\tanh(\beta \lambda)\sinh(\lambda))^2-1}}=C\theta(x)
\end{align*}
On the other hand, the derivative of $\tilde{\Lambda}$ is given by 
\begin{equation*}
 \partial_t \tilde{\Lambda} = \dfrac{\varepsilon\theta^{\prime}\left(|v|\tanh(\gamma)-\frac{1}{\sqrt{c}}\beta\sech^2(\beta\gamma)\tanh(\gamma)-\frac{1}{\sqrt{c}}\tanh(\beta\gamma)\right)}{\sech(\gamma)\sqrt{\left(|v|\cosh(\gamma)-\frac{1}{\sqrt{c}}\sinh(\gamma)\tanh(\beta\gamma)\right)^2-1}},
 \end{equation*}
in this case $\gamma:= \lambda+\varepsilon\theta,$ then, is clearly that $ \partial_t \tilde{\Lambda}|_{\{t=0\}} \in C_0^{2}(\mathbb{R})$.

\bigskip

 Next, for the field $\hat{\phi}$, we have a bounded function and $\hat{\phi}(t,x) >0$ for all $t,x\in \mathbb{R}$. Again, since $\theta\in C_0^{2}(\mathbb{R})$ we have,
 \begin{equation*}
 	\hat{\phi}(t=0,x)=C_1(\lambda), \quad  \forall x\in \mathbb{R}\setminus \supp \theta,
 \end{equation*}	
 	 and  we can define 
\[
\phi(t,x)=\hat{\phi}-\epsilon \quad \mbox{with} \quad	\epsilon=C_1(\lambda).
\]
With this definition one has:
	\begin{equation*}
		\phi(t=0,x)=\hat{\phi}(t=0,x)-\epsilon,
	\end{equation*}
then
\begin{equation*}
	\phi(t=0,x)=\hat{\phi}(t=0,x)-\epsilon=0 \quad \forall x\in \mathbb{R}\setminus \supp \theta,
\end{equation*}
 if we choose 
 \begin{equation*}
 	\varepsilon\phi_0(x)=\phi(t=0,x)=\hat{\phi}(t=0,x)-\epsilon,
 \end{equation*}
 where $\phi_0$ is given as:
	\begin{align*}
\phi_0(x)&=	-\dfrac{1}{2\varepsilon}\arctan\left(\cosh(\beta(\lambda+\varepsilon\theta(x)))\cosh(\lambda+\varepsilon \theta(x))(\tanh(\beta(\lambda+\varepsilon \theta(x)))+v\sqrt{c}\tanh(\lambda+\varepsilon \theta(x)))\right)\\
& \quad +\dfrac{1}{2\varepsilon}\arctan\left(\cosh(\beta\lambda)\cosh(\lambda)(\tanh(\beta\lambda)+v\sqrt{c}\tanh(\lambda))\right), 	
	\end{align*}
this definition is again, a suitable consideration, we can compute the $\lim\limits_{\varepsilon \longrightarrow 0}\phi_0$ using L'H\^opital's rule:
\begin{align*}
	\lim\limits_{\varepsilon \longrightarrow 0}\phi_0&=\lim\limits_{\varepsilon \longrightarrow 0} \frac{\theta(x)((1+\beta v \sqrt{c})\sinh(\gamma)\sinh(\beta\gamma)+(\beta+v\sqrt{c})\cosh(\gamma)\cosh(\beta\gamma)   }{2(1+(\cosh(\gamma)\sinh(\beta \gamma)+\beta \sqrt{c}\sinh(\gamma)\cosh(\beta\gamma))^2)}\\
	& \quad = C_1\theta(x).
\end{align*} 
The function $\phi$ have the desired local and global well-posedness properties. Indeed, the derivative of this function is:
\begin{equation*}
	\partial_t \phi= \dfrac{-\varepsilon\theta^{\prime}(\beta+v\sqrt{c}+(1+\beta v\sqrt{c})\tanh(\beta\gamma)\tanh(\gamma))}{2\sech(\beta\gamma)\sech(\gamma)(\left(\cosh(\gamma)\sinh(\beta \gamma)+v \sqrt{c}\sinh(\gamma)\cosh(\beta\gamma)\right)^2+1)},
\end{equation*}
which is also a localized function. Finally from the previous analysis, we can conclude that for $\partial_t\Lambda,\partial_t\phi \in L^2(\mathbb{R})$, with $\Lambda(t,x)=\hat{\Lambda}(t,x),$ then 
 \begin{equation*}
	E[\Lambda,\phi] < \infty. 
\end{equation*}
In the end, $\hat{\Lambda}$ reads as
\[
\hat{\Lambda}=\ln(\cosh(\gamma))+ \ln\left(|v|-\frac{\tanh(\gamma)\tanh(\beta\gamma)}{\sqrt{c}}+\sqrt{ \left( |v|-\frac{\tanh(\gamma)\tanh(\beta\gamma)}{\sqrt{c}}\right)^2-1}\right).
\]
This finishes the proof.
\end{proof}		

%Some final comments before the end. In the special case where
%	\begin{equation*}
%		\theta(x):= \left\{ \begin{array}{cc}
%			\exp\left(-\frac{1}{1-|x|^2}\right), & |x|< 1 \\
%			0, & |x|\geq 1,
%		\end{array}\right.
%	\end{equation*}
%one has the following data (see Figure \ref{solutions}):
%\begin{figure}[h]
%	\centering
%	\subfigure[$\tilde{\Lambda}(t=0,x)$]{
%		\includegraphics[width=0.3\textwidth]{Lambda1Tilde.png}}
%	\subfigure[$\phi(t=0,x)$]{
%		\includegraphics[width=0.3\textwidth]{Phi_0.png}}
%	\subfigure[$\partial_t\tilde{\Lambda}|_{\{t=0\}}$]{
%		\includegraphics[width=0.3\textwidth]{Lambda1t.png}}
%	\subfigure[$\partial_t\phi|_{\{t=0\}}$]{
%		\includegraphics[width=0.3\textwidth]{Phi1t.png}}
%	\caption{Initial data.}\label{solutions}
%\end{figure}
% 
%

%%%%%%%%%%%%%%%%%%%%%%%%%%%%%%%%%%%%%%%%%%%%%%%%%%%%%%%%%%%%%%%%%%
		\appendix
	\section{Some useful inequalities}\label{A}
This section start by presenting the well-known  Gronwall's lemma:
	\begin{lem} Let $f: \mathbb{R} \longrightarrow \mathbb{R}$ be a positive continuous function and $g: \mathbb{R} \longrightarrow \mathbb{R}$ be a positive integrable function
		such that
		\begin{equation*}
			f(t)\leq A + \int_{0}^{t} f(s)g(s)ds,
		\end{equation*}
		for some $A\geq 0$ for every $t\in [0,T]$. Then
		
		\begin{equation}
			f(t)\leq A\exp\left(\int_{0}^{t}g(s)ds\right),
		\end{equation}
		for every  $t\in [0,T]$.
	\end{lem}
	The second result to be presented is related to another pointwise bounds that were presented for Luli et. al. in \cite{Luli2018} for the study of the global problem in the Section \refeq{Sect:3}:
	\begin{lem}\label{lemaaux}
	Under the assumption (\refeq{supuesto}) and (\refeq{supuesto1})	exits a universal constant $C_3$ so that
	
	\begin{equation*}
	\begin{aligned}
	&~{} \norm{\frac{\varphi(u)^{1/2}}{\varphi(\underline u)^{1/4}} \underline{L}\tilde{\Lambda}(t,x) }_{L^{\infty}(\Sigma_t)} \\
	&~{} \qquad \leq C_3 \left( \norm{\frac{\varphi(u)^{1/2}}{\varphi(\underline u)^{1/4}} \underline{L}\tilde{\Lambda}(t,x) }_{L^{2}(\Sigma_t)} + \norm{\frac{\varphi(u)^{1/2}}{\varphi(\underline u)^{1/4}}\underline{L}\partial_x \tilde{\Lambda}(t,x)}_{L^{2}(\Sigma_t)}  \right),
	\end{aligned}
	\end{equation*}
	
		\begin{equation*}
		\begin{aligned}
	&~{} \norm{\frac{\varphi(\underline u)^{1/2}}{\varphi(u)^{1/4}} L\tilde{\Lambda}(t,x) }_{L^{\infty}(\Sigma_t)} \\
	&~{} \qquad  \leq  C_3 \left( \norm{\frac{\varphi(\underline u)^{1/2}}{\varphi(u)^{1/4}} L\tilde{\Lambda}(t,x) }_{L^{2}(\Sigma_t)} + \norm{\frac{\varphi(\underline u)^{1/2}}{\varphi(u)^{1/4}}L\partial_x \tilde{\Lambda}(t,x)}_{L^{2}(\Sigma_t)}  \right),
	\end{aligned}
	\end{equation*}
and
	\begin{equation*}
	\begin{aligned}
	&~{}
	\norm{\frac{\varphi(u)^{1/2}}{\varphi(\underline u)^{1/4}} \underline{L}\phi(t,x) }_{L^{\infty}(\Sigma_t)} \\
	&~{} \qquad \leq C_3 \left( \norm{\frac{\varphi(u)^{1/2}}{\varphi(\underline u)^{1/4}} \underline{L}\phi(t,x) }_{L^{2}(\Sigma_t)} + \norm{\frac{\varphi(u)^{1/2}}{\varphi(\underline u)^{1/4}}\underline{L}\partial_x \phi(t,x)}_{L^{2}(\Sigma_t)}  \right),
	\end{aligned}
\end{equation*}

\begin{equation*}
\begin{aligned}
	&~{}
	\norm{\frac{\varphi(\underline u)^{1/2}}{\varphi(u)^{1/4}} L\phi (t,x) }_{L^{\infty}(\Sigma_t)} \\
	&~{} \qquad\leq  C_3 \left( \norm{\frac{\varphi(\underline u)^{1/2}}{\varphi(u)^{1/4}} L\phi(t,x) }_{L^{2}(\Sigma_t)} + \norm{\frac{\varphi(\underline u)^{1/2}}{\varphi(u)^{1/4}}L\partial_x \phi(t,x)}_{L^{2}(\Sigma_t)}  \right).
	\end{aligned}
\end{equation*}
	\end{lem}
%%%%%%%%%%%

{\color{blue} \section{Ending of proof of Theorem \ref{GLOBAL0}}\label{B}
	In this section, we describe the details of the estimates for the second equation in \eqref{problema2} that complete the proof of Theorem \ref{GLOBAL0}.  
	%\subsection{Ending of the Theorem \ref{GLOBAL0}} 
	
	\medskip
	
	For simplicity, in Section \ref{Sect:3} we worked  with the first equation of system \eqref{problema2}. Now we prove the estimates for the second equation. 
	\begin{proof}  The first step is the following:  Using (\refeq{NC22}) and (\refeq{derivadanullC}) in the second equation of (\refeq{problema2}), we obtain:
	\begin{equation}\label{derivadas1}
		 \square \left( \partial_x \phi\right)  = 2\Big[\coth(\lambda +\tilde{\Lambda})\left(Q_0(\partial_x\phi,\tilde \Lambda)+Q_0(\phi,\partial_x\tilde \Lambda)\right)-2\partial_x\tilde{\Lambda}\csch^2(\lambda+\tilde{\Lambda})Q_0(\phi,\tilde \Lambda)\Big].
	\end{equation}
As in Section \ref{Sect:3}, fix $\delta \in (0,1),$ under the assumptions (\refeq{supuesto})-(\refeq{supuesto1})-(\refeq{condicionlambda}) for all $t\in [0,T^{*}]$, we assume that the solution remains regular, to later show that these bounds are maintained, with a better constant. 

\medskip

Consider $k=0,1$. Using \eqref{EEnergia} on \eqref{derivadas}, with $\psi=\partial^k_x \phi$, and taking the sum over $k=0,1$, we obtain
\begin{equation}\label{EE3}
	\begin{aligned}
		&\overline{\mathcal{E}}(t)+\overline{\mathcal{F}}(t) \leq 2C_0 \overline{\mathcal{E}}(0)\\
		&\quad +2C_0  \iint_{D_t} \left( (1+\left|u \right|^2)^{1+\delta} | \underline{L}\phi|+ (1+\left|\underline{u} \right|^2)^{1+\delta}|L \phi|\right)2 |\coth(\lambda +\tilde{\Lambda})| |Q_0(\phi,\tilde \Lambda)| \\
		&\quad        +4C_0\iint_{D_t} \left( (1+\left|u \right|^2)^{1+\delta} | \underline{L}\partial_x\phi|+ (1+\left|\underline u \right|^2)^{1+\delta}| L\partial_x \phi|\right) |\coth(\lambda+\tilde{\Lambda})||(Q_0(\partial_x\phi, \tilde \lambda)+ Q_0(\phi,\partial_x\tilde \Lambda) )| \\
				& \quad +4C_0 \iint_{D_t} \left( (1+\left|u \right|^2)^{1+\delta} | \underline{L}\partial_x\phi|+ (1+\left|\underline u \right|^2)^{1+\delta}| L\partial_x \phi|\right) |\partial_x\tilde{\Lambda} \csch^2(\lambda +\tilde{\Lambda})||Q_0(\phi,\tilde \Lambda)| \\
		& =: 	2C_0\overline{\mathcal{E}(0)}+ 2C_0\sum_{j=1}^8 I_j,
	\end{aligned}
\end{equation}
 In this case, the integrals $I_j, i\in \{1,2,..,8\}$ are defined as follows: 
	
\begin{equation}
\begin{aligned}
& I_1:= 2C_0  \iint_{D_t} \left( (1+\left|u \right|^2)^{1+\delta} | \underline{L}\phi|\right) |\coth(\lambda +\tilde{\Lambda})| |Q_0(\phi,\tilde \Lambda)| \\
& I_2:= 2C_0 \iint_{D_t} \left( (1+\left|\underline{u} \right|^2)^{1+\delta}|L \phi|\right)|\coth(\lambda +\tilde{\Lambda})| |Q_0(\phi,\tilde \Lambda)| \\
& I_3: = 2C_0\iint_{D_t} \left( (1+\left|u \right|^2)^{1+\delta} | \underline{L}\partial_x\phi|\right)  |\coth(\lambda+\tilde{\Lambda})||(Q_0(\partial_x\phi, \tilde \Lambda)|\\
& I_4:= 2C_0\iint_{D_t} \left( (1+\left|u \right|^2)^{1+\delta} | \underline{L}\partial_x\phi|\right)  |\coth(\lambda+\tilde{\Lambda})||Q_0(\phi,\partial_x\tilde \Lambda) )| \\
& I_5: = 2C_0\iint_{D_t} \left( (1+\left|\underline u \right|^2)^{1+\delta} L\partial_x\phi|\right)  |\coth(\lambda+\tilde{\Lambda})||(Q_0(\partial_x\phi, \tilde \Lambda)|\\
& I_6:= 2C_0\iint_{D_t} \left( (1+\left|\underline u \right|^2)^{1+\delta} | L\partial_x\phi|\right)  |\coth(\lambda+\tilde{\Lambda})||Q_0(\phi,\partial_x\tilde \Lambda) )| \\
& I_7:= 2C_0  \iint_{D_t} \left( (1+\left|u \right|^2)^{1+\delta} | \underline{L}\partial_x\phi|\right) |\partial_x\tilde{\Lambda} \csch^2(\lambda +\tilde{\Lambda})||Q_0(\phi,\tilde \Lambda)| \\
& I_8:= 2C_0  \iint_{D_t} \left( (1+\left|\underline u \right|^2)^{1+\delta}| L\partial_x \phi|\right) |\partial_x\tilde{\Lambda} \csch^2(\lambda +\tilde{\Lambda})||Q_0(\phi,\tilde \Lambda)|.
\end{aligned}
\end{equation}
The goal is to control the right-hand side of the above estimate. Essentially we have eight terms to control, but several are equivalent  and we only need to consider essentially two cases. Indeed, it will be sufficient to bound the terms corresponding to $\underline L \partial_x \phi$ and $\underline L \phi,$ since by symmetry, the procedure for the other terms will be analogous. First, we start to bound the  term $I_7,$ that represents the most attention, given that it has different sub-terms to estimate, recalling that we define $\varphi (x)= (1+|x|^2)^{1+\delta},$ with $0<\delta \ll 1.$\\

Taking into account \eqref{importante}, \eqref{condicionlambda} and \eqref{expansion1}-\eqref{cotas}, and writing $\partial_x \tilde \Lambda= \frac{1}{2}(L-\underline L )\tilde \Lambda $, we get
\begin{equation}\label{I7}
\begin{aligned}
I_7 \lesssim &~{} C_0\iint_{D_t} \varphi(u) |\underline L \partial_x \phi| |L \tilde \Lambda||L \phi| |\underline L \tilde \Lambda| + C_0\iint_{D_t} \varphi(u) |\underline L \partial_x \phi| |L \tilde \Lambda|^2|\underline L \phi|  \\
& +C_0\iint_{D_t} \varphi(u) |\underline L \partial_x \phi| |L \tilde \Lambda|^2|L \phi| + C_0\iint_{D_t} \varphi(u) |\underline L \partial_x \phi| |\underline L \tilde \Lambda||\underline L \phi| | L \tilde \Lambda| \\
:= & ~{} I_{7,1}+ I_{7,2}+I_{7,3}+I_{7,4}.
\end{aligned}
\end{equation}
Recall that by Fubini's Theorem the spacetime $D_t$ in \eqref{D_t} is foliated by $\underline C_{\underline u}$ for $u\in \mathbb{R},$ and also by $\{t\}\times \Sigma_t, t\in \mathbb{R}.$ Using again the Lemma \ref{lema1}, we obtain
\begin{align*}
	  I_{7,1} \lesssim & \iint_{D_t} \varepsilon \underbrace{(\varphi(\underline u)^{-3/4}\varphi(u)^{1/2}|\underline{L}\partial_x \phi|)}_{L_t^2L_x^2}\underbrace{(\varphi^{1/2}(\underline u)|L\tilde \Lambda|)}_{L_t^{\infty}L_x^2}\underbrace{(\varphi(\underline u)^{-1/4}\varphi (u)^{1/2}|\underline{L}\tilde \Lambda|)}_{L_t^2 L_x^{\infty}}\\
	& \lesssim \varepsilon \underbrace{\left(\iint_{D_t} \dfrac{\varphi (u)|\underline{L}\partial_x \phi|^2}{\varphi(\underline u)^{3/2}} \right)^{1/2}}_{T_1} 
	\underbrace{\sup_{t\in [0,T^*]}\left(\int_{\Sigma_t}\varphi(\underline u)|L \tilde \Lambda|^2 \right)^{1/2}}_{T_2} 
	\underbrace{\left(  \int_{0}^{t} \norm{ \dfrac{\varphi (u)^{1/2}}{\varphi(\underline u)^{1/4}} |\underline{L}\tilde \Lambda|}_{L^{\infty}(\Sigma_{\tau})}^{2} d\tau  \right)^{1/2}.}_{T_3}	
\end{align*}
Let us study each of the factors $T_j$. For $T_1$, one has:
\begin{align*}
	T_1^2 \leq & \int_{\mathbb{R}}\left[\int_{\underline{C}_{\underline u}} \dfrac{\varphi (u)|\underline{L}\partial_x\phi|^2}{\varphi(\underline u)^{3/2}} ds\right]d\underline{u} = \int_{\mathbb{R}} \dfrac{1}{\varphi  (\underline u)^{3/2}}\underbrace{\left[\int_{\underline{C}_{\underline u}} \varphi (u)|\underline{L}\partial_x \phi|^2ds \right]}_{\lesssim \overline{\mathcal{F}}_1(t)} d\underline{u}
	\lesssim  \int_{\mathbb{R}} \dfrac{\varepsilon^2}{\varphi  (\underline u)^{3/2}}d\underline{u},
\end{align*} 
since the integral is finite, we have
$T_1 \lesssim \varepsilon.$
 The integral $T_2$ is part of the energy norm  $\mathcal{E}_0(t)$ in \eqref{energies} then $T_2 \lesssim \varepsilon.$
For the term $T_3$ one can use the same argument as in \cite{Luli2018}: using Lemma \refeq{lemaaux} one gets
\begin{align*}
	T_3 &\lesssim  \left(\int_0^t  \norm{\frac{\varphi(u)^{1/2}}{\varphi(\underline u)^{1/4}} \underline{L}\tilde \Lambda (t,x) }^2_{L^{2}(\Sigma_{\tau})} +\int_0^t  \norm{\frac{\varphi(u)^{1/2}}{\varphi(\underline u)^{1/4}} \underline{L}\partial_x\tilde \Lambda(t,x) }_{L^{2}(\Sigma_{\tau})}^2  \right)^{1/2} \\
	&\lesssim \left( \iint_{D_t} \frac{\varphi(u)}{\varphi(\underline u)^{1/2}}|\underline{L}\tilde \Lambda|^2 + \iint_{D_t} \frac{\varphi(u)}{\varphi(\underline u)^{1/2}}|\underline{L}\partial_x \tilde \Lambda|^2  \right)^{1/2},
\end{align*}
both terms above are of the same form as $T_1$
 and then we have that $T_3 \lesssim \varepsilon.$ We conclude  that $  I_{7,1} \lesssim  \varepsilon^4.$ 
 \medskip
 Now we control the integral $I_{7,2}$ in \eqref{I7}, using again the Lemma \ref{lema1}, the assumption \eqref{supuesto1} and Cauchy–Schwarz inequality. We have:
 \begin{align*}
 I_{7,2}=& C_0\iint_{D_t} \varphi(u) |\underline L \partial_x \phi| |L \tilde \Lambda|^2|\underline L \phi|  \leq \iint_{D_t}C_2 \varepsilon^2 \frac{\varphi(u)^{1/2}}{\varphi(\underline u)^{1/2}}|\underline L \partial_x \phi| \frac{\varphi(u)^{1/2}}{\varphi(\underline u)^{1/2}}|\underline L \phi|\\
 \leq & C_2\varepsilon^2  \left(  \iint_{D_t}  \frac{\varphi(u)}{\varphi(\underline u)}|\underline L \partial_x \phi|^2  \right)^{1/2} \left(  \iint_{D_t}  \frac{\varphi(u)}{\varphi(\underline u)}|\underline L \phi|^2  \right)^{1/2} \lesssim \varepsilon^4.
 \end{align*}
 To finish with the term $I_7$ we need to estimate the terms $I_{7,3}$ and $I_{7,4}$ in \eqref{I7}, which are similar in structure, for this case we get:
 
 \begin{align*}
 I_{7,34}= & ~{} I_{7,3}+I_{7,4}\lesssim   \iint_{D_t} \varepsilon^2 |\underline{L}\partial_x \phi||L\phi|+  \iint_{D_t} \varepsilon^2 |\underline{L}\partial_x \phi||L\tilde \Lambda|\\
	=& \iint_{D_t} \varepsilon^2 \dfrac{\varphi (u)^{1/2}}{\varphi(\underline u)^{1/2}}|\underline{L}\partial_x\phi |\left( \dfrac{\varphi(\underline u)^{1/2}}{\varphi (u)^{1/2}}|L\phi| + \dfrac{\varphi(\underline u)^{1/2}}{\varphi (u)^{1/2}}|L\tilde \Lambda|\right)\\
	\lesssim & \iint_{D_t} \varepsilon^2 \left( \dfrac{\varphi (u)}{\varphi(\underline u)}|\underline{L}\partial_x\phi|^2 \right) +  \iint_{D_t} \varepsilon^2\left(\dfrac{\varphi(\underline u)}{\varphi (u)}|L\phi|^2 \right) +  \iint_{D_t} \varepsilon^2\left(\dfrac{\varphi(\underline u)}{\varphi (u)}|L\tilde \Lambda|^2 \right)\\
	\lesssim & \int_{\mathbb{R}} \dfrac{\varepsilon^2}{\varphi(\underline u)}\underbrace{\left[ \int_{\underline{C}_{\underline u}} \varphi (u) |\underline{L}\partial_x\phi |^2ds \right]}_{\lesssim \overline{\mathcal{F}}_1(t)} d\underline{u} + \int_{\mathbb{R}} \dfrac{\varepsilon^2}{\varphi(u)}\underbrace{\left[ \int_{C_{u}} \varphi (\underline u) |L\phi|^2ds \right]}_{\lesssim \overline{\mathcal{F}}_0(t)} du\\
	& +  \int_{\mathbb{R}} \dfrac{\varepsilon^2}{\varphi(u)}\underbrace{\left[ \int_{C_{u}} \varphi (\underline u) |L\tilde \Lambda|^2ds \right]}_{\lesssim \mathcal{F}_0(t)} du \lesssim \varepsilon^4.
 \end{align*}
 
Putting all estimates together for $I_7,$ we conclude that $I_7 \lesssim \varepsilon^4.$ A similar result is obatained for $I_8.$

\medskip Now we treat the term $I_1+I_3+I_4$ from \eqref{EE3}. We have from \eqref{importante} and  \eqref{expansion1}-\eqref{cotas},

\begin{align*}
& \iint _{D_t} \varphi(u ) | \underline{L}\partial_x\phi| \left(  |L\partial_x\phi ||\underline{L}\tilde \Lambda|+ |\underline L \partial_x \phi| |L\tilde \Lambda| + |L\phi||\underline L \partial_x \tilde \Lambda |+ |\underline L \phi||L \partial_x \tilde \Lambda |\right)\\
 &\quad \quad \quad +\iint_{D_t} \left( \varphi(u)| \underline{L}\phi|\right)\left( |L\phi ||\underline{L}\tilde \Lambda|+ |L\tilde \Lambda| |\underline{L}\phi|\right).
\end{align*}
Using the condition (\refeq{condicionlambda}), the situation matches Case 1 developed in \cite{Luli2018}.  All these integrals can be written as
\begin{align*}
	\sim  \iint_{D_t}  \left(\varphi (u) |\underline{L}\partial_x \tilde{\Lambda}||L\phi||\underline{L}\partial_x \phi| + \varphi (u) |\underline{L}\partial_x \tilde{\Lambda}||L\partial_x \phi| |\underline{L}\phi|\right).
\end{align*}
Then, we can use the estimate \eqref{bounds} in Section \ref{Sect:3} to conclude the bounds on these terms, which again are of order $\varepsilon^3.$

\end{proof}
%%%%%%%%%%%%%%%%%%%%%%
\section{Classical Solution: Local Theory}\label{C}
As we can see, Proposition \ref{LOCAL} does not directly provide us with a classical solution for the initial value problem \eqref{cauchy2}. In order to obtain such a classical solution, we need an initial data with sufficient regularity, which allows us to control the terms associated with the nonlinearity. The idea of the proof still has the same structure.\\

Recall that the initial value problem for (\refeq{cauchy2}) can be written in vector form as follows
	 \begin{equation}
	 	\begin{cases*}
	 		\partial_{\alpha} (m^{\alpha \beta}\partial_{\beta}\Psi)=F(\Psi,\partial \Psi)\\
	 		(\Psi,\partial_t \Psi)|_{\{t=0\}}=(\Psi_0, \Psi_1) \in  \mathcal{\hat H}.
	 	\end{cases*}\label{IVPC}
	 \end{equation}
	 Where $m^{\alpha \beta}$ are the components of the Minkowski metric with $\alpha, \beta \in \left\{0,1\right\}$, and 
	 \begin{equation}\label{Hcal}
	(\Psi,\partial_t \Psi)\in \mathcal{\hat H}:=H^3(\mathbb{R})\times H^{3}(\mathbb{R}) \times H^2(\mathbb{R}) \times H^2(\mathbb{R}).
	 \end{equation}
	 %Notice that from \eqref{tilda}, $\Lambda\in \dot H^1$. 
	  We are also going to impose the following condition on the initial data  
	 \begin{equation}\label{condicioninicial}
	 	\norm{\left(\Psi_0,\Psi_1\right)}_{\mathcal{\hat H}} \leq \dfrac{\lambda}{2D},
	 \end{equation}
	 where the assumptions on the constant $D\geq 1$ will be indicated below. 	 
	 \medskip
	 
	 The following proposition shows that the equation (\refeq{IVPC}), in terms of the function $\tilde{\Lambda}$ introduced in \eqref{tilda}, is locally well-posed in the space $L^{\infty}([0,T]; \mathcal{\hat H})$ with the norm in this space defined by 
	 	\begin{equation*}
	 		\norm{ (\Psi,\partial_t \Psi) }_{L^{\infty}([0,T]; \mathcal{H})} = \sup_{t\in [0,T]} \left( \norm{ \Psi }_{H^3(\mathbb{R})\times H^{3}(\mathbb{R})}+ \norm{ \partial_t \Psi }_{H^2(\mathbb{R})\times H^2(\mathbb{R})} \right),
	 	\end{equation*} 
with $(\Psi,\partial_t \Psi)$ introduced in \eqref{notation}. The result is the following.
	 
	 \begin{prop}\label{LOCAL1}
	 	If $(\Psi_0, \Psi_{1})$ satisfies the condition  (\ref{condicioninicial}) with an appropriate constant $D\geq1$, then: 
	 	\begin{itemize}
	 		\item[(1)] (Existence and uniqueness of local-in-time solutions). There exists 
	 		\[ 
			T=T\left( \norm{ \left( \tilde{\Lambda}_0, \phi_0 \right) }_{H^3(\mathbb{R}) \times H^3(\mathbb{R})}, \norm{ \left( \tilde{\Lambda}_1, \phi_1 \right) }_{H^2(\mathbb{R}) \times H^2(\mathbb{R})},\lambda \right) > 0,\]
	 		such that there exists a (classical) solution $\Psi$ to  (\refeq{IVPC}) with 
	 		\begin{equation*}
	 			(\Psi,\partial_t \Psi)\in L^{\infty}([0,T];\mathcal{\hat H}).
	 		\end{equation*}
	 		Moreover, the solution is unique in this function space.
			
			\medskip
			
	 		\item[(2)] (Continuous dependence on the initial data). Let $\Psi_{0}^{(i)}, \Psi_{1}^{(i)}$ be sequence such that $\Psi_{0}^{(i)} \longrightarrow \Psi_{0}$ in $H^3(\mathbb{R})\times H^{3}(\mathbb{R})$ and $\Psi_{1}^{(i)} \longrightarrow \Psi_{1}$ in $H^2(\mathbb{R})\times H^2(\mathbb{R})$ as $i \longrightarrow \infty.$ Then taking $T>0$ sufficiently small, we have 
	 		\begin{equation*}
	 			\norm{ \left(\Psi^{(i)}-\Psi, \partial_t(\Psi^{(i)}-\Psi) \right) }_{L^{\infty}([0,T]; H^s(\mathbb R)\times  H^{s}(\mathbb R) )\times  L^{\infty}([0,T]; H^{s-1}(\mathbb R) \times  H^{s-1}(\mathbb R))} \longrightarrow 0.
	 		\end{equation*}	 
	 		as $i \longrightarrow \infty$ for every $1\leq s < 3.$ Here $\Psi$ is the solution arising from data $(\Psi_0,\Psi_1)$ and $\Psi^{(i)}$ is the solution arising from data $\left(\Psi_0^{(i)},\Psi_1^{(i)} \right).$
	 	\end{itemize}	 	
	 \end{prop}

 \begin{proof}[Proof of Proposition \ref{LOCAL1}] 
	 	(1). This part of the Proposition is proved by Picard's iteration. Using a density argument it is sufficient to assume the initial data $(\Psi_0,\Psi_1)\in \mathcal{S}^4$ ($\mathcal S$ being the Schwartz class), along with condition (\refeq{condicioninicial}). Define a sequence of smooth functions $\Psi^{(i)},$ with $i \geq 1$ such that 
		\[
		\Psi^{(1)}=(0,0),
	 	\]
	 	and for $i \geq 2,$ $\Psi^{(i)}$ is iteratively defined as the unique solution to the system
	 	\begin{equation}
	 		\begin{cases*}
	 			\partial_{\alpha} (m^{\alpha \beta}\partial_{\beta}\Psi^{(i)})=F(\Psi^{(i-1)},\partial \Psi^{(i-1)})\\
	 			(\Psi^{(i)},\partial_t \Psi^{(i)})|_{\{t=0\}}=(\Psi_0, \Psi_1) \in  \mathcal{H}.
	 		\end{cases*}\label{IVP1Ci}
	 	\end{equation}
		 It is important to note that from \eqref{notation} and \eqref{condicioninicial} we can assure that for $j=1,2,$
	\begin{equation}\label{condicion2f}
		\sum_{ \gamma =0}^{2} \sup_{|x|,|p|\leq \frac{\lambda}{2}}|\partial^{\gamma}_{x,p} F_j|(x,p) \leq C_{j,\frac12\lambda}.
	\end{equation} 
	Indeed, this can be seen from the fact that for $(x,p)=(x_1,x_2,p_1,p_2,p_3,p_4)$ and $|x|\leq \frac{\lambda}2,$
	\[
	F_1(x,p)= 2\sinh(2\lambda +2x_1)\left(p_4^2-p_3^2\right),\quad F_2(x,p)= \dfrac{\sinh(2(\lambda +x_1))}{\sinh^2(\lambda +x_1)} \left(p_3 p_1 - p_2p_4 \right).
	\]
	Define bounded functions in the class $C^1$.
	
	\medskip
	
	 	It is important to note that condition (\refeq{condicion2f}) allows this iterative definition of the functions $\Psi^{(i)}$ to be possible, since it maintains each component of $F$ with the required regularity, see \cite{sogge}. First, it will be shown that for a sufficiently small $T>0,$ the sequence $(\Psi,\partial_t \Psi)$ is uniformly (in $i$) bounded in $L^{\infty}([0,T]; \mathcal{\hat H})$, then it will be shown that it is also a Cauchy sequence. For the first part, the idea is to use the energy estimates (\refeq{Eenergia}), we want to prove that there is a constant $ 0 < A \leq \frac{\lambda}{2}$ such that 
	 	\begin{equation}\label{HI1}
	 		\norm{ \left( \Psi^{(i-1)},\partial_t \Psi^{(i-1)} \right) }_{L^{\infty}([0,T];\mathcal{\hat H})} \leq A,	
	 	\end{equation}
	 	implies that 
	 	\begin{equation*}
	 		\norm{ \left(\Psi^{(i)},\partial_t \Psi^{(i)} \right) }_{L^{\infty}([0,T];\mathcal{\hat H})} \leq A.	
	 	\end{equation*} 
	 	The energy estimation (\refeq{Eenergia}) allows us to write for (\refeq{IVPC}) the following estimate:  
	 	\begin{equation}
	 		\begin{aligned}
	 			 \sup_{t\in [0,T]} \norm{\left(\Psi^{(i)},\partial_t \Psi^{(i)}\right)}_{\mathcal{\hat H}} 
	 			 	&~{}\leq  C(1+T)(\norm{\left(\Psi_0, \Psi_1\right)}_{\mathcal{\hat H}})  \\
	 			 +C(1+T)	&~{} \int_{0}^{T} \left(\norm{ F_1\left(\Psi^{(i-1)},\partial \Psi^{(i-1)}\right)}_{H^2(\mathbb{R})}+\norm{ F_2\left( \Psi^{(i-1)},\partial \Psi^{(i-1)} \right)}_{H^2(\mathbb{R})}\right)(t)dt.
	 		\end{aligned} 
	 	\end{equation}
	 	With this estimate, our goal is to bound the integral on the right hand side of the inequality above. That is,  we want to prove that there exists $B=B(A,F)>0$ such that for $t\in [0,T],$  we have 
		\begin{equation}
		\sum_{n=0}^{2} ||\partial_x^nF(\Psi^{(i-1)},\partial \Psi^{(i-1)})||_{L^2}(t) \leq B.
		\end{equation}
		For this, we will use the conditions (\refeq{condicioninicial}) for each $F_j$ which is satisfied by the hypothesis in  (\refeq{HI1}), if $B_1=\max \{C_{1,\frac{\lambda}{2}}, C_{2,\frac{\lambda}{2}} \},$ and  using  chain rule we get 		
		\begin{equation*}
		\begin{aligned}
		\sum_{n=0}^{2} ||\partial_x^nF(\Psi^{(i-1)},\partial \Psi^{(i-1)})||_{L^2} \leq &~{} B_1 + B_1||\partial_x \Psi^{i-1}||_{L^2}+ B_1||\partial \partial_x \Psi^{i-1}||_{L^2}+ B_1||\partial \Psi^{i-1}||^2_{H^2}\\
		& + B_1||\partial_x^2 \Psi^{i-1}||_{L^2}+ B_1||\partial \partial_x \Psi^{i-1}\cdot \partial \partial_x \Psi^{i-1} ||_{L^2}+ ||\partial \partial_x^2 \Psi^{i-1}||_{L^2} \\
		\leq & ~{} B,
		\end{aligned}
		\end{equation*}
		where $B=B(B_1,A,\lambda)$, which results in the following estimate
	 	\begin{equation}
	 		\begin{aligned}
	 			\sup_{t\in [0,T]} \norm{ \left(\Psi^{(i)},\partial_t \Psi^{(i)} \right)}_{\mathcal{\hat H}} \leq  C(1+T)\left( \norm{ \left(\Psi_0, \Psi_1 \right)}_{\mathcal{\hat H}}+2BT \right),
	 		\end{aligned} 
	 	\end{equation}
	 	we can choose $T> 0$  sufficiently small such that 
	 	\begin{equation*}
	 		2BT \leq \norm{\left(\Psi_0,\Psi_1\right)}_{\mathcal{\hat H}},
	 	\end{equation*} 
	 	so
	 	\begin{equation*}
	 	\norm{ \left(\Psi^{(i)},\partial_t \Psi^{(i)}\right)}_{L^{\infty}([0,T]; \mathcal{\hat H})} \leq 2C \norm{ (\Psi_0,\Psi_1) }_{\mathcal{\hat H}}.
	 	\end{equation*}
	 	If we choose $D > 4C$ in (\refeq{condicioninicial})  and  $A:= 2C||(\Psi_0,\Psi_1)||_{\mathcal{\hat H}}\leq \frac{2C\lambda}{D} \leq \frac{\lambda}{2}.$ We have thus shown the desired implication.
		\\
		
	 	In the Section \ref{Sect:2} we showed that the last sequence is of Cauchy type in the larger space $L^{\infty}([0,T]; \mathcal{H})$. Therefore, the sequence is Cauchy on $L^{\infty}([0,T]; \mathcal H),$ and hence convergent.  That is, there exists $(\Psi, \partial_t \Psi)$ in $L^{\infty}([0,T]; \mathcal H)$. The uniform bounds  (on $i$) in $L^{\infty}([0,T], \mathcal{\hat H})$  guarantees that the limit in fact lies in the smaller space $L^{\infty}([0,T], \mathcal{\hat H}),$  that is, for almost $t\in [0,T],$  $(\Psi^{(i)}, \partial_t \Psi ^{(i)})(t)\in \mathcal{\hat H}$, uniform in $i$, and therefore by Banach-Alaoglu's Theorem, there is a weak limit in  $ \mathcal{\hat H}$ (up to a subsequence). But  the uniqueness of the limit ensures that this limit must agree with  $(\Psi, \partial_t \Psi)(t).$ This concludes the proof of existence.
		
	 	\medskip
		
	 	Finally, for the continuous dependence on initial data, we prove in the Section \ref{Sect:2} that taking $i \longrightarrow \infty,$ we get
	 	\begin{equation*}
	 		\sup_{s\in [0,t]} \norm{ \left(\Psi^{(i)}-\Psi, \partial_t \Psi^{(i)}-\partial_t \Psi \right)}_{{\color{blue}\mathcal{H}}} \longrightarrow 0. 
	 	\end{equation*}	
		To obtain the result in general for $1 \leq s < 3 ,$ simply observe that 
		\begin{equation*}
		\begin{aligned}
		&\sup_{t\in [0,T]} ||(\Psi^{(i)}-\Psi,\partial_t \Psi^{(i)}-\partial_t \Psi)||_{H^s\times H^{s}\times H^{s-1}\times H^{s-1}}(t)  \\
		& \quad \leq C\sup_{t\in [0,T]}\left( ||(\Psi^{(i)}-\Psi,\partial_t \Psi^{(i)}-\partial_t \Psi)||_{H^1\times H^{1}\times L^{2}\times L^{2}}(t) \right)^{\frac{3-s}{2}}\\
		& \quad \quad \times \left( ||(\Psi^{(i)}-\Psi,\partial_t \Psi^{(i)}-\partial_t \Psi)||_{H^3\times H^{3}\times H^{1}\times H^{1}}(t) \right)^{\frac{s-1}{2}} \longrightarrow 0.
		\end{aligned}
		\end{equation*}
		This last property ends the proof of Proposition \ref{LOCAL}. 
	 \end{proof}
}
	\bigskip
	
	%%%%%%%
	\bibliographystyle{unsrt}
	%\bibliography{bibfile}
	
	%\begin{thebibliography}{99}
	
	%\bibitem{ABC} ABC
	
	%\bibitem{DFG} DFG
	
	%\end{thebibliography}
	
\end{document}